\documentclass[11pt,a4paper,leqno]{amsart}
\usepackage{amsmath,amsfonts,amssymb,amsthm}
\usepackage[english]{babel}
\usepackage{xcolor}
\usepackage{datetime}
\usepackage{enumerate}

\usepackage[colorlinks=true,linkcolor=blue]{hyperref}%

\newtheorem{thm}{Theorem}[section]

\newtheorem{cor}[thm]{Corollary}
\newtheorem{lemma}[thm]{Lemma}

\newtheorem{prop}[thm]{Proposition}
\numberwithin{equation}{section}

\newtheorem{theoremletra}{Theorem}

\newtheorem{theorem}{Theorem}

\theoremstyle{definition}
\newtheorem{remark}[thm]{Remark}

\def\R{\mathbb{R}}
\def\N{\mathbb{N}}
\def\Z{\mathbb{Z}}

\def\P{\mathbf{P}}
\def\E{\mathbf{E}}
\def\V{\mathbf{V}}

\definecolor{verde}{rgb}{0.0, 0.5, 0.0}

\newcommand{\uno}{\mathop{\textbf{\Large 1}}}

\begin{document}

\title{Some arithmetic properties of P\'{o}lya's urn}
\date{\today}

\author{Jos\'{e} L. Fern\'{a}ndez}
\address{Departamento de Matem\'{a}ticas, Universidad Aut\'{o}noma  de Madrid, 28049 Madrid, Spain.}
\email{joseluis.fernandez@uam.es}

\author{Pablo Fern\'{a}ndez}
\address{Departamento de Matem\'{a}ticas, Universidad Aut\'{o}noma  de Madrid, 28049 Madrid, Spain.}
\email{pablo.fernandez@uam.es}

\thanks{Jos\'{e} L. Fern\'{a}ndez and Pablo Fern\'{a}ndez are partially supported by Fundaci\'{o}n Akusmatika.}

\begin{abstract}
 Following Hales (2018), the evolution of P\'{o}lya's urn may be interpreted as a walk, a P\'{o}lya walk, on the integer lattice $\mathbb{N}^2$.
 We study the visibility properties of P\'{o}lya's  walk or, equivalently, the  divisibility  properties of the composition of the urn. In particular,
 we are interested in  the asymptotic average time  that a P\'{o}lya walk  is visible from the origin, or, alternatively,  in the asymptotic proportion
 of draws so that the resulting composition of the urn is coprime. Via de Finetti's exchangeability theorem,  P\'{o}lya's walk appears as a mixture of standard
 random walks. This paper is a follow-up of Cilleruelo--Fern\'{a}ndez-Fern\'{a}ndez (2019),
  where similar questions were studied for standard random walks.
\end{abstract}


	\maketitle

\begin{flushright}
\it In memoriam Javier Cilleruelo.
\end{flushright}



\section{Introduction} In this paper we extend
some  results of Cilleruelo, Fern\'{a}ndez and Fern\'{a}ndez, \cite{CFF}, which connect standard random walks in $\N^2$ with the visibility of points in $\N^2$ (or with the coprimality of pairs of positive integers), to the so called P\'{o}lya walk in $\N^2$.

The approach of this paper consists on expressing  P\'{o}lya's walk, via de Finetti's theorem, see Theorem \ref{teor:deFinetti}, as a mixture of standard $\alpha$-random walks in $\N^2$ with $\alpha$ ranging in the interval $(0,1)$. Visibility properties of standard $\alpha$-random walks were studied in~\cite{CFF}. A basic  tool, there and now,  is the second moment method, see Proposition \ref{prop:Borel}. To apply this method for P\'{o}lya's walk,  under the mixture representation, we need to determine the dependence upon $\alpha$ of the error terms of the  estimates of \cite{CFF} for the $\alpha$-random walks, which in that paper were inmaterial.

 \subsection{P\'{o}lya's walk in $\N^2$} Following Hales in \cite{Hales},  the \textit{P\'{o}lya walk} is a discrete time stochastic process $(\mathbb{Y}_n)_{n \ge 0}$ with values in the lattice $\N^2$ which geometrically codifies the evolution of P\'{o}lya's urn.

 The walk starts at some given deterministic initial  position  $\mathbb{Y}_0=(a_0,b_0)\in \N^2$.
For each $n\ge 0$, the jump of the walk $\mathbb{Y}_{n+1}-\mathbb{Y}_n$ can take only two values, $(1,0)$ and $ (0,1)$: the walk either  moves one unit upwards or one unit rightwards; for each $n \ge 0$, \textit{given} the position $\mathbb{Y}_n=(a_n,b_n)$ of the walk at time $n$, the conditional probabilities of the only two admissible  jumps are
 \smallskip
 \begin{equation*}\P\big(\mathbb{Y}_{n+1}-\mathbb{Y}_n=(1,0)\big)=\frac{a_n}{a_n+b_n}\quad \mbox{and} \quad
 \P\big(\mathbb{Y}_{n+1}-\mathbb{Y}_n=(0,1)\big)=\frac{b_n}{a_n+b_n}\,\cdot\end{equation*}

(Along this paper, probabilities, expectations  and variances of random variables on diverse underlying spaces will be denoted simply  by~$\P$, $\E$ and~$\V$.)

The stochastic evolution of the process $(\mathbb{Y}_n)_{n \ge 0}$ is determined solely by the starting position $(a_0,b_0)$. The random coordinates $(a_n,b_n)$ of the walk $\mathbb{Y}_n$ satisfy $a_n+b_n=n+a_0+b_0$, for each $n \ge 1$.

The coordinates $(a_n,b_n)$ of $\mathbb{Y}_n$ starting at $(a_0,b_0)$ register the composition of a \textit{standard P\'{o}lya's urn} at time $n$ with an initial composition consisting of $a_0$ amber balls and~$b_0$ blue balls: successively draw a ball uniformly at random from the urn, notice the color, return the drawn ball to the urn, and add \textit{one}  ball of the same color. Thus $a_n$ is the  number of amber balls and $b_n$ is the number of blue balls in the urn, at time $n$, i.e., after $n$ successive random drawings.

Thus, the pair $(a_n,b_n)$ registers both the position of P\'{o}lya's walk after $n$ steps or the composition of P\'{o}lya's urn after $n$ drawings. Along this paper, we will alternate between these two interpretations (although we will favour the first one).

\subsection{Visible points}

A point $(x,y)$ in  $\N^2$ is called \textit{visible from the origin}, or simply visible, if no point of $\N^2$ other than $(x,y)$ itself lies in the straight segment that joins~$(0,0)$ and $(x,y)$.

We use $\gcd(x,y)$ to denote the greatest common divisor of positive integers $x$ and $y$. In divisibility terms, $(x,y)\in \N^2$ is visible from $(0,0)$  if and only if  $\gcd(x,y)=1$, i.e., if the integer  coordinates $x$ and $y$ are coprime integers. It is always this characterization of visibility in terms of coprimality of coordinates that we will use in the estimates and calculations which follow.

\smallskip

Let $\mathcal{V}$ denote the set of points $(x,y)\in \N^2$ which are \textit{visible} from the origin.

A classical theorem of Dirichlet, originally in \cite{Dir}, but see also Hardy--Wright \cite{HW}, Section 18.3, and there, in particular, Theorem 332  (and also Sections~\ref{seccion:Dirichlet's density}, 
\ref{sec:extension 1 de Dirichlet} and~\ref{seccion:general density result} in this paper), gives the  density of the set $\mathcal{V}$:
\begin{equation}\label{eq:dirichlet density of visible}
\lim_{N \to \infty} \frac{1}{N^2} \,\#\big(\mathcal{V}\cap \{1, \ldots, N\}^2\big)=\frac{1}{\zeta(2)}= \frac{6}{\pi^2}\, \cdot\end{equation}

In probabilistic terms, if $N$ is large, the probability that two integers $x$ and $y$ drawn  independently with uniform distribution from $\{1, \ldots, N\}$ is approximately $1/\zeta(2)$.

\smallskip

For a natural number $k\ge 1$, a point $(x,y)\in \N^2$ is called \textit{$k$-visible} if $\gcd(x,y)=k$. This means that in the segment from $(0,0)$ to $(x,y)$ there are exactly $k$ points of $\N^2$, counting $(x,y)$ as one of them, namely, $(jx/k,jy/k)$, for $1\le j \le k$.

We denote the set of those points in $\N^2$ which are \textit{$k$-visible} from the origin with $\mathcal{V}_k$. Observe that $\mathcal{V}_1=\mathcal{V}$ and that for integer $k\ge 1$, we have that $\mathcal{V}_k=k \mathcal{V}$, in the sense that $(x,y)\in \mathcal{V}_k$ if and only there exists $(x',y')\in \mathcal{V}$ such that $x=kx'$ and $y=ky'$.

From $\mathcal{V}_k=k \mathcal{V}$  it follows that the  density of $\mathcal{V}_k$ is given by
\[
\lim_{N \to \infty} \,\frac{1}{N^2} \#\big(\mathcal{V}_k\cap \{1, \ldots, N\}^2\big)=\frac{1}{k^2\,\zeta(2)}\, \cdot\]

%
%
%

Let $I$ be the indicator function of the set $\mathcal{V}$, thus $I(x,y)=1$ if $\gcd(x,y)=1$ and $I(x,y)=0$ otherwise; and for each integer $k \ge 1$,  let $I_k$ be the indicator function of $\mathcal{V}_k$. Observe that $I_1\equiv I$.

\subsection{Visits of P\'{o}lya's walk to visible points}

In this paper we are mostly interested in the  asymptotic proportion of time that  P\'{o}lya's walk spends on the set $\mathcal{V}$ of visible points.

For integer $N\ge 1$, we denote with 
\begin{equation}\label{eq:def QN} Q_N\triangleq \frac{1}{N} \sum_{n=1}^N I(\mathbb{Y}_n)\end{equation}
the random variable that registers the proportion of time (or steps) up to time~$N$ (excluding $\mathbb{Y}_0$) that P\'{o}lya's walk remains visible from the origin.

It turns out that, almost surely, the sequence of random proportions $Q_N$ converges to a limit as $N\to \infty$, \textit{and} that  almost surely this limit is a constant independent of the drawing of the walk; in fact, this limit is almost surely the (asymptotic) density $1/\zeta(2)$ of the set $\mathcal{V}$ of visible points.

We state these facts, the main result of this paper, as:
\begin{theorem}\label{teor:asymptotic coprime proportion of polya walks} For any given initial position $(a_0,b_0)\in \N^2$,
\[
\lim_{N \to \infty} Q_N=\frac{1}{\zeta(2)} \quad \mbox{almost surely}\, .
\]
\end{theorem}
Irrespectively of the initial position, it is always the case that, almost surely and asymptotically, the proportion of time that  P\'{o}lya's  walk spends on the set $\mathcal{V}$ of visible points is $1/\zeta(2)$.

In terms of P\'{o}lya's urn, Theorem  \ref{teor:asymptotic coprime proportion of polya walks} claims that \textit{almost surely, the proportion of drawings resulting in  number of amber balls and number of blue balls  which are coprime converges to $1/\zeta(2)$}.

\smallskip

For $k$-visibility, we have analogously:
\begin{theorem}\label{teor:asymptotic gcd=k proportion of polya walks} For any given initial condition $(a_0,b_0)\in \N^2$, one has that, for each $k \ge 1$,
\[
\lim_{N \to \infty} \frac{1}{N} \sum_{n=1}^N I_k (\mathbb{Y}_n )=\frac{1}{k^2\,\zeta(2)} \quad \text{almost surely}.
\]
\end{theorem}
Thus, almost every evolution  of  P\'{o}lya's walk spends \textit{for each $k \ge 1$} an asymptotic proportion  $1/(k^2\zeta(2))$ of time on the set $\mathcal{V}_k$.

\begin{remark}
We would like to point out that in \cite{Hales}, Hales is mainly interested  in  two other stochastic processes:  Farey's walk and P\'{o}lya's variant walk. Contrary to the P\'{o}lya walk which we are considering, these processes visit exclusively visible points, and their main appeal lies in their interesting ergodic properties.
\end{remark}

\subsubsection{Several balls added at each drawing; larger steps in the walk}
More generally, consider the P\'{o}lya walk with initial position $(a_0, b_0)\in\mathbb{N}^2$ and with up-step  and right-step of integer size $c\ge 1$. In terms of the urn model, we start with $a_0$ amber balls and~$b_0$ blue balls, and we add $c$ amber balls (instead of just 1) if the drawn ball is amber, and~$c$ blue balls otherwise. We denote by $\mathbb{Y}^c_n$ the position at step $n\ge 0$ of the walk with parameters $(a_0,b_0)$ and $c$. We have the following.



\begin{theorem}\label{teor:polya from (1,1) step c}
For integers $a_0, b_0\ge 1$ and $c \ge 1$ and the P\'{o}lya walk $(\mathbb{Y}^c_n)_{n \ge 0}$ with up-step $(0,c)$ and right-step $(c,0)$, and starting from $(a_0,b_0)$,
there exists a constant $\Delta(a_0,b_0;c,c)$ $($explicit, and computable$)$  such that
\[\lim_{N \to \infty} \frac{1}{N} \sum_{n=1}^N I(\mathbb{Y}^c_n)=\Delta(a_0,b_0;c,c)\quad \mbox{almost surely}\, .\]
\end{theorem}

Theorem \ref{teor:polya from (1,1) step c} appears in Section~\ref{section:Changing the steps of the walk} as Theorem~\ref{teor:polya from (1,1) step c, bis} for the particular case where $(a_0, b_0)=(1,1)$,
and as Theorem~\ref{teor:polya from (a_0,b_0) steps generales} for the general case.

This constant $\Delta(a_0,b_0;c,c)$ is given in \eqref{eq:formula for big Delta con mu y c}. The reason of the double appearance of~$c$ is because it is a particular instance of a more general expression $\Delta(a_0,b_0;r_0,u_0)$, to be considered later on Section \ref{seccion:general density result}; there, a couple of explicit formulas for $\Delta(a_0,b_0;r_0,u_0)$ can be found. For instance, for the  case $(a_0,b_0)=(1,1)$ and $u_0=r_0=c$, the expression of $\Delta(1,1;c,c)$ is as explicit as
\[
\Delta(1,1;c,c)=\sum_{\substack{d \ge 1,\\ \gcd(d,c)=1}} \frac{\mu(d)}{d^2}=\prod_{p \, \nmid \, c} \Big(1-\frac{1}{p^2}\Big)\, .
\]
Here, as  usual, $\mu$ stands for the M\"{o}bius function.
The product above $\prod_{p \, \nmid \, c} $ extends to all primes $p$ which do not divide $c$. For $c=1$, the product extends  to all primes and $\Delta(1,1;1;1)=\prod_{p}(1-1/p^2)={1}/{\zeta(2)}$, as claimed in Theorem~\ref{teor:asymptotic coprime proportion of polya walks}. But, for instance, $\Delta(1,1;2;2)=\frac{4}{3}\, \frac{1}{\zeta(2)}$.
Observe that $\Delta(1,1;c,c)$ depends only on the prime factors of $c$. 



\subsection{Standard random walks in $\N^2$}

Let $\alpha \in (0,1)$. Consider the $\alpha$-random walk  $(\mathbb{Z}_{\alpha,n})_{n\ge 0}$  with values in $\N^2$ starting at some initial position $\mathbb{Z}_{\alpha,0}=(a_0,b_0)$ with \textit{independent and identically distributed} increments given by
\[\P\big(\mathbb{Z}_{\alpha, {n+1}}-\mathbb{Z}_{\alpha, n}=(1,0)\big)=\alpha\quad \text{and}\quad
 \P\big(\mathbb{Z}_{\alpha,{n+1}}-\mathbb{Z}_{\alpha, n}=(0,1)\big)=(1-\alpha)\,.
 \]

 The $\alpha$-random walk moves up by $(0,1)$ with probability $(1-\alpha)$ and to the right by $(1,0)$ with probability $\alpha$.

 \smallskip

 For $\alpha\in (0,1)$ and integer $N\ge 1$, we denote
\begin{equation}\label{eq:def SNa} S_{\alpha,N}\triangleq\frac{1}{N} \sum_{n=1}^N I(\mathbb{Z}_{\alpha,n})\, .\end{equation}
This $S_{\alpha,N}$ is a random variable which registers the proportion of time (or steps) up to time $N$ that the $\alpha$-random walk remains visible from the origin.

\smallskip

For the $\alpha$-random walk $\mathbb{Z}_{\alpha,n}$, it has been shown in \cite{CFF} that:

 \begin{theoremletra}\label{teor:basic asymptotic of CFF} For any $\alpha \in (0,1)$,  the $\alpha$-random walk $\mathbb{Z}_{\alpha,n}$ starting at $(0,0)$ satisfies
 \[
\lim_{N \to \infty} S_{\alpha,N}=\frac{1}{\zeta(2)}\, , \quad \mbox{almost surely}\,.\]
 \end{theoremletra}

 Actually, Theorem \ref{teor:basic asymptotic of CFF} holds for any starting position $(a_0,b_0)\in \mathbb{N}^2$ of the walk, and, moreover, almost surely and for each $k\ge 1$,
 \[
 \lim_{N \to \infty} \frac{1}{N} \sum_{n=1}^N I_k(\mathbb{Z}_{\alpha, n})=\frac{1}{k^2\zeta(2)}\,\cdot
 \]
Theorem \ref{teor:basic asymptotic of CFF} is  stated in \cite{CFF} for starting point $(0,0)$ and mere visibility, and not for arbitrary starting point in $\N^2$ of the walk and $k$-visibility. We will discuss that extension in Section \ref{section:gcd=k}, see Theorem \ref{teor:basic asymptotic of CFF with k}.

\subsection{Notations and some preliminaries} For a set $A$, we use $\#A$ to denote the cardinality of $A$.

For a subset $B$ of $\N^2$, we call the limit
\[
\lim_{N\to \infty}\frac{1}{N^2} \#\big(B\cap \{1, \ldots, N\}^2\big)\triangleq D(B)
\]
(if it exists) the density $D(B)$ of $B$.

We use $\gcd(x,y)$ and $\text{\rm lcm}(x,y)$ to denote, respectively,  the greatest (positive) common divisor and the least 
 common multiple of the positive integers $x$ and $y$.

As usual, $\mu$ and $\phi$ denote, respectively,  the M\"obius and the Euler totient functions, and $\zeta$ stands for the Riemann zeta function. With $\tau(n)$ we denote the number of divisors of $n$.


The symbol $\prod_{p}$ means an infinite product which extends over all primes $p$, while $\prod_{p\in \mathcal{P}}$ means a product (maybe infinite) which extends over all primes satisfying property $\mathcal{P}$.

Probabilities, expectations  and variances of random variables on diverse underlying spaces will be denoted simply  by~$\P$, $\E$ and $\V$. For  random variables $X$ and $Y$, the notation   $X\stackrel{\rm d}{=}Y$ signifies that  $X$ and $Y$ have the same distribution: $\P(X\in B)=\P(Y\in B)$, for any Borel set in $\R$.

With $X\sim \textsc{bin}(N,p)$, for integer $N\ge 1$ and $p\in (0,1)$, we signify that the distribution of the random variable $X$  is binomial with $N$ repetitions and probability of success $p$.
Besides, $X\sim\textsc{beta}(a,b)$ means that the distribution of the random variable $X$ is a beta distribution with parameters $a,b\in (0,\infty)$. See Section \ref{subsection:beta}.

The de Finetti mixture (probability Borel) measure associated to an exchangeable sequence of Bernoulli variables is denoted by $\nu$.

With $C, C^\prime, \ldots$, we denote some absolute constants.

\subsubsection{Some number-theoretical results}\label{subseccion:number-theoretical}
We shall use several times the fact that $\gcd(a,b)=\gcd(a,a+b)$. For the M\"{o}bius function $\mu$, we have that
\begin{equation}
\label{eq:sum mu/d2}
\sum_{d=1}^{\infty}\frac{\mu(d)}{d^2}=\frac{1}{\zeta(2)}=\frac{6}{\pi^2}\,,
\end{equation}
and further that
\begin{equation}
\label{eq:sum mu/d2entera}
\sum_{d\le n}\frac{\mu(d)}d\, \Big\lfloor \frac{n}{d}\Big\rfloor =\frac{6n}{\pi^2}+O(\log n)\,,\quad\text{as $n\to\infty$.}
\end{equation}
The identity
\begin{equation}
\label{eq:Mobius y Euler}
\sum_{d\mid k}\frac{\mu(d)}{d}=\frac{\phi(k)}{k}\,,
\end{equation}
where $\phi$ denotes Euler's totient function, follows by M\"obius inversion from the Gauss identity $n=\sum_{d \mid n}\phi(d)$, for $n \ge 1$.

For $\tau(n)$, the number of divisors of $n$, we have
\begin{equation}
\label{eq:order of tau}
\tau(n)=O(n^{\delta})\quad\text{for all $\delta >0$, as $n\to\infty$.}
\end{equation}
%
%
See Theorem 315
 in \cite{HW}. In this paper, taking $\delta =1/4$ will be good enough, as we just shall need an exponent less than $1/2$.

For $k\ge 1$, we let $g_k$ denote the arithmetic function given by $g_k(n)=1$ if $n$ is a multiple of $k$, and $g_k(n)=0$, otherwise. We let $\delta_k$ be the Kronecker delta: $\delta_k(n)=1$ if $n=k$, and $\delta_k(n)=0$ otherwise.

Thus $g_k(n)=\sum_{d\mid n} \delta_k(d)$, for each $n \ge 1$, while, from M\"{o}bius inversion, we see  that $\delta_k(n)=\sum_{kd\mid n} \mu(d)$, for each $n \ge 1$; and we deduce also that,
 for $k \ge 1$,
\begin{equation}
\label{eq:representation of indicator of gcd}
\delta_k(\gcd(n,m))=\sum_{kd \mid n,\, kd \mid m}\mu(d)\,, \quad \mbox{for each $n,m \ge 1$}\,.
\end{equation}

\subsubsection{Dirichlet's density theorem}\label{seccion:Dirichlet's density}

Observe that
\[\begin{aligned}\sum_{1\le n,m\le N}\delta_1(\gcd(n,m))&=\sum_{1\le n,m\le N}\sum_{d\mid n,\, d\mid m} \mu(d)\\&=
\sum_{d\ge 1} \mu(d)\,\#\big\{1\le n,m\le N: d|n, d|m\big\}=\sum_{d\ge 1} \mu(d)\Big\lfloor \frac{N}{d}\Big\rfloor^2\,.\end{aligned}\]
Writing
$\lfloor {N}/{d}\rfloor= {N}/{d}-\{ {N}/{d}\}$,  and using that $|\mu(d)|\le 1$ for any $d \ge 1$, we obtain that
\[\sum_{1\le n,m\le N}\delta_1(\gcd(n,m))=N^2 \sum_{d \ge 1}\frac{\mu(d)}{d^2}+O(N\ln N)\,,\quad\text{as $N\to\infty$}.
\]
In particular, using \eqref{eq:sum mu/d2}, we see that
\[
\frac{\#\big\{1\le n,m\le N: \gcd(n,m)=1\big\}}{N^2}=\frac{1}{N^2} \sum_{1\le n,m\le N}\delta_1(\gcd(n,m))\xrightarrow{N\to\infty}\sum_{d \ge 1}\frac{\mu(d)}{d^2}=\frac{6}{\pi^2}\,,\]
which is Dirichlet's density result (with error estimate): the density of coprime pairs is $1/\zeta(2)$. This is, of course,  standard, but we have recalled it since we shall use later on, see Sections~\ref{sec:extension 1 de Dirichlet} and~\ref{seccion:general density result}, a bit more elaborate version of that type of argument.

\subsubsection{Beta distributions}\label{subsection:beta} We shall need and use the $\textrm{beta}$ distributions of probability in the interval $(0,1)$. Here we recall definitions, establish notations and register some simple facts.

We shall denote the density of the $\textsc{beta}(a,b)$ distribution with parameters $a,b>0$ by $f_{a,b}(\alpha)$; thus
\begin{equation}\label{eq:formula density fab}f_{a,b}(\alpha)=\frac{\alpha^{a-1} (1-\alpha)^{b-1}}{\mathrm{Beta}(a,b)}\, , \quad \mbox{for $\alpha \in (0,1)$}\, ,\end{equation}
where $\mathrm{Beta}$ denotes the beta function:
\begin{equation}\label{eq:def beta function}
\mathrm{Beta}(a,b)=\int_0^1 \alpha^{a-1} (1-\alpha)^{b-1}\, d\alpha=\frac{\Gamma(a)\,\Gamma(b)}{\Gamma(a+b)}\, ,
\end{equation}
and $\Gamma$ is the gamma function.

The expected value of a $\textsc{beta}(a,b)$ distribution is $a/(a+b)$.
For $a=1$, $ b=1$, the $\textsc{beta}(1,1)$ distribution is just the uniform distribution in $(0,1)$.

It will relevant later on to observe that
\begin{equation}\label{eq:integral finita Beta con raices}
\int_0^1 \frac{1}{\sqrt{\alpha (1-\alpha)}} \,f_{a,b}(\alpha) \,d\alpha<\infty
\end{equation}
if (and only if) $a>1/2$ and $b>1/2$. And that
\begin{equation}\label{eq:integral finita Beta sin raices}
\int_0^1 \frac{1}{{\alpha (1-\alpha)}} \,f_{a,b}(\alpha) \,d\alpha<\infty
\end{equation}
if (and only if) $a>1$ and $b>1$.

\smallskip

In what follows, mostly integer values of the parameters $a$ and $b$ would intervene, and  in that case we have that
\[
\Gamma(a)=(a-1)!\quad \text{and}\quad \mathrm{Beta}(a,b)=\frac{a+b}{ab}\Big/\binom{a+b}{a}\, .
\]

\subsection{Plan of the paper} Section \ref{section:exchange and Polya} recalls the  de Finetti exchangeability theorem and applies it to the P\'{o}lya walk.

Section \ref{section:proof of main} gives a proof of the main theorem of this paper, Theorem \ref{teor:asymptotic coprime proportion of polya walks}, based on the mixture representation and modulo some precise estimates for $\alpha$-random walks which are dealt with in Section \ref{section:estimates for random walk}.

Section \ref{section:gcd=k} is devoted to the study of $k$-visibility, while Section \ref{section:Changing the steps of the walk} exhibits the main results for walks with steps of general size and arbitrary starting point.

\section{Exchangeability and P\'{o}lya's walk}\label{section:exchange and Polya}

We first describe in Section \ref{section:exchange and de Finetti} some known facts about the de Finetti exchangeability theorem and establish some notation and terminology.
Next, in Section \ref{section:exchange polya urn walk}, we apply the exchangeability theorem to P\'{o}lya's urn and walk, recalling along the way some basic properties of P\'{o}lya's urn. Finally, in Section \ref{section:polya walk as mixture}, we exhibit P\'{o}lya's walk as a mixture of $\alpha$-standard random walks.

For P\'{o}lya's urn, the original sources are \cite{EggPolya} and \cite{Polya};  but see \cite{JK}. As for de Finettis's theorem, it first appeared in \cite{deF} and \cite{deF2}; but see \cite{Aldous}, \cite{HeathSudderth}, \cite{Kingman} and \cite{Kirsch}.

\

\subsection{Exchangeability and de Finetti's theorem}\label{section:exchange and de Finetti}

Let $(G_n)_{n\ge 1}$ be a sequence of Bernoulli random variables defined all of them in the same probability space.

The sequence $(G_n)_{n\ge 1}$ is said to be \textit{exchangeable} if for any integer $n \ge 1$ and any list $(y_1, \ldots, y_n)$ extracted from $\{0,1\}$ and any permutation $\sigma$ of $(1, \ldots, n)$ we have that
\[
\P(G_1=y_1,\ldots, G_n=y_n)=\P(G_{\sigma(1)}=y_1, \ldots, G_{\sigma(n)}=y_n)\,.\]
In this definition, either the random variables can be permuted, as above, or the values.

\smallskip

The following theorem of de Finetti is the fundamental result on exchangeable sequences of Bernoulli variables.

\begin{theoremletra} \label{teor:deFinetti} Let $(G_{j})_{j\ge 1}$ be an {\upshape infinite} exchangeable sequence of Bernoulli random variables defined in some probability space. Then we have the following.
\begin{enumerate}[\rm (1)]
\item There exists a unique Borel probability measure $\nu$ on the interval $(0,1)$ so that, for any $n \ge 1$ and any list $(y_1,\ldots, y_n)$ extracted from $\{0,1\}$, it holds that
\[
\P(G_1=y_1, \ldots, G_n=y_n)=\int_0^1 \theta^{s_n} (1-\theta)^{n-s_n} \,d \nu(\theta)
\]
where $s_n$ denotes the sum $s_n=\sum_{j=1}^n y_j$, that is, the number of $y_j$ which are equal to $1$.

\item The frequency limit
\[\lim_{n\to \infty} \frac{1}{n} \sum_{j=1}^n G_j\triangleq M \quad \mbox{exists almost surely}\]
and defines a random variable $M$ whose law is precisely $\nu${\rm:}
\[\P(M\in B)=\nu(B)\,, \quad \mbox{for any Borel set $B\subset (0,1)$}\,.\]

\item Conditioning upon any  value $\theta$ of the limit $M$, the variables $G_{j}$ are independent and, for any $n\ge 1$,
\[
\P\big(G_1=y_1, \ldots, G_n=y_n\, |\, M=\theta\big)=\theta^{s_n} \,(1-\theta)^{n-s_n}\,.\]
\end{enumerate}
\end{theoremletra}
The probability measure $\nu$ is known as the \textit{de Finetti mixture measure} for the exchangeable sequence~$(G_{j})_{j\ge 1}$.

\smallskip

Observe that from (1) of Theorem \ref{teor:deFinetti} it follows that, for $n \ge 1$ and $0\le k \le n$,
\begin{equation}
\label{eq:distribution sumGj}
\P\Big(\sum_{j=1}^n G_j=k\Big)=\int_0^1 \binom{n}{k} \theta^k (1-\theta)^{n-k} d\nu(\theta)=
\int_0^1 \P(\textsc{bin}(n, \theta)=k)\, d\nu(\theta)\,,
\end{equation}
and, moreover, that from (3) of Theorem \ref{teor:deFinetti} it follows that
\begin{equation}
\label{eq:distribution sumGj bis}
\P\Big(\sum_{j=1}^n G_j=k\,\big|\, M=\theta\Big)=\P (\textsc{bin}(n, \theta)=k )\,, \quad \mbox{for $n \ge 1$ and $0\le k \le n$}\,.
\end{equation}
Thus, conditioning on the value $\theta$ of $M$,  $(G_j)_{j\ge 1}$ is a sequence of independent and identically distributed variables, actually, Bernoulli variables with parameter $\theta$.

Unconditionally, each variable $G_j$ is a Bernoulli variable with parameter $\int_0^1 \theta d\nu(\theta)$.

\smallskip

Observe that statement (1) of Theorem \ref{teor:deFinetti} is in fact a characterization of exchangeability: any  sequence $(G_j)_{j\ge 1}$ satisfying  (1) of Theorem \ref{teor:deFinetti} is exchangeable. This is so because the probabilities $\P(G_1=y_1, \ldots, G_n=y_n)$ depend only on the sum $s_n=\sum_{j=1}^n y_j$ (the number of 1's among the $y_j$), and not on the order of the $y_j$.

\subsection{Exchangeability and P\'{o}lya's urn and walk}\label{section:exchange polya urn walk}

Next we study P\'{o}lya's urn and P\'{o}lya's  walk from the perspective of exchangeability. In what follows, the urn starts with a composition of $a_0$ amber balls and $b_0$ blue balls, or the walk starts at $(a_0, b_0)$.

For each $j \ge 1$, we let $F_j$ be the Bernoulli variable which takes the value $1$ if the ball added at stage $j$ is amber, and $F_j=0$ if the ball added at stage $j$ is blue. In terms of the walk,  $F_j=1$ if $\mathbb{Y}_j-\mathbb{Y}_{j-1}=(1,0)$.

For $n\ge 1$, we have   that $\sum_{j=1}^n F_j=a_n-a_0$ for the urn, and  that
\begin{equation}\label{eq:jumps in terms of sumFj}
\mathbb{Y}_n-\mathbb{Y}_0=\Big(\sum_{j=1}^n F_j, n-\sum_{j=1}^n F_j\Big)\,,\end{equation}
for the walk.

\smallskip

For  integer $n \ge 1$ and for $x_1, \ldots, x_n\in \{0,1\}$, and with $t_n=\sum_{j=1}^n x_j$, it follows readily from conditioning and by induction that
\begin{equation}
\label{eq:exchange Fn1}
\P(F_1=x_1, \ldots, F_n=x_n)
=\frac{\prod_{j=0}^{t_n-1} (a_0+j)\, \prod_{j=0}^{n-t_n-1} (b_0+j)}{\prod_{j=0}^{n-1} (a_0+b_0+j)}\,\cdot
\end{equation}
The exchangeability of the sequence $(F_j)_{j\ge 1}$ follows from \eqref{eq:exchange Fn1} since the probability in there only depends on the sum $t_n$ of the $x_j$ and not on the order of the $x_j$.

We can rewrite \eqref{eq:exchange Fn1}, in terms of Beta functions, and using \eqref{eq:def beta function}, as
\begin{align}
\nonumber
\P(F_1=x_1, \ldots, F_n=x_n)
&=\frac{\Gamma(a_0+t_n)}{\Gamma(a_0)}\,
\frac{\Gamma(b_0+n-t_n)}{\Gamma(b_0)}\,
\frac{\Gamma(a_0+b_0)}{\Gamma(a_0+b_0+n)}
\\ \label{eq:exchange Fn}
&=\frac{\textrm{Beta}(a_0+t_n, b_0+n-t_n)}{\textrm{Beta}(a_0,b_0)}\,\cdot
\end{align}
Or even more compactly, using the expression for the density $f_{a_0,b_0}(\alpha)$ in \eqref{eq:formula density fab}, as
%
\begin{equation}\label{eq:exchange Fn bis}
\P(F_1=x_1, \ldots, F_n=x_n)=\int_0^1 \alpha^{t_n}(1-\alpha)^{n-t_n} f_{a_0,b_0}(\alpha) \,d\alpha
\end{equation}
This formula \eqref{eq:exchange Fn bis} shows, in particular,  that the de Finetti mixture measure for the exchangeable sequence $(F_j)_{j\ge 1}$ is given by $d\nu(\alpha)=f_{a_0,b_0}(\alpha) \,d\alpha$, i.e., it is a  $\textsc{beta}(a_0,b_0)$ distribution.

The variables $F_j$ are Bernoulli variables with a common probability of success:
\[\P(F_j=1)=\frac{a_0}{a_0+b_0}=\int_0^1 \alpha f_{a_0, b_0}(\alpha) \, d\alpha\,, \quad \mbox{for each $j \ge 1$}\,;\]
but they are not independent.

De Finetti's Theorem \ref{teor:deFinetti} applied to the sequence $(F_j)_{j\ge 1}$ gives  us further that
\[\lim_{n\to \infty} \frac{1}{n}\sum_{j=1}^n F_j\triangleq L \quad \mbox{exists almost surely}\]
and that this limit defines a random variable $L$ with values in $(0,1)$ whose distribution is, by virtue of  \eqref{eq:exchange Fn bis}, precisely $d\nu(\alpha)=f_{a_0,b_0}(\alpha) \,d\alpha$, i.e., the limit variable $L$ is a $\textsc{beta}(a_0,b_0)$ random variable. Moreover,  conditioning on a limit value $\alpha$ of $L$ we have that
\[
\P(F_1=x_1, \ldots, F_n=x_n\,|\, L=\alpha)=\alpha^{t_n} (1-\alpha)^{n-t_n}\,, \quad \mbox{for any $\alpha\in (0,1)$}\,.
\]

Now, formulas \eqref{eq:distribution sumGj} and \eqref{eq:distribution sumGj bis} applied to the jumps $\mathbb{Y}_n-\mathbb{Y}_0$ of P{\' o}lya's walk written as in \eqref{eq:jumps in terms of sumFj} give us that, for $n \ge 1$ and $0\le k\le n$,
\[
\P\big(\mathbb{Y}_n-\mathbb{Y}_0=(k,n-k)\big)=\int_0^1 \P(\textsc{bin}(n,\alpha)=k) \,f_{a_0,b_0}(\alpha)\, d\alpha\,,
\]
and that, for $n \ge 1$ and $0\le k\le n$ and $\alpha \in (0,1)$,
\[
\P\big(\mathbb{Y}_n-\mathbb{Y}_0=(k,n-k)\,|\, L=\alpha\big)= \P(\textsc{bin}(n,\alpha)=k) \,.
\]

\subsubsection{Limit distribution of the slope of the walk}

We consider P\'{o}lya's walk starting from an initial position $(a_0,b_0)\in \mathbb{N}^2$.

We have  $\mathbb{Y}_n-\mathbb{Y}_0=(a_n-a_0, b_n-b_0)=\big(\sum_{j=1}^n F_j, n-\sum_{j=1}^n F_j\big)$, for $n \ge 1$.

Write also $\mathbb{Y}_n-\mathbb{Y}_0$, for $n \ge 1$, in polar coordinates: $\mathbb{Y}_n-\mathbb{Y}_0=r_n (\cos \psi_n, \sin \psi_n)$ with $\psi_n \in (0,\pi/2)$ and $r_n \ge 1$. Both $r_n$ and $\psi_n$ are random variables, for $n \ge 1$.

Since $\frac{1}{n}\sum_{j=1}^n F_j$ tends to $L$ almost surely, we see that $(a_n-a_0)/n$ and $(b_n-b_0)/n$ tend almost surely to $L$ and $1-L$, respectively, and so
\[
\frac{r_n}{n}=\sqrt{\frac{(a_n-a_0)^2}{n^2}+\frac{(b_n-b_0)^2}{n^2}}\to\sqrt{L^2+(1-L)^2}\,, \quad \mbox{almost surely as $n \to \infty$}\,,\]
and, also,
\[\tan \psi_n=\frac{b_n-b_0}{a_n-a_0}=\frac{(b_n-b_0)/n}{(a_n-a_0)/n}\to \frac{1-L}{L}\quad \mbox{almost surely as $n \to \infty$}\,.\]

Thus, the random angle
$\psi_n$ tends, as $n \to \infty$ almost surely to a variable $\Psi$ given by $\Psi=\arctan \tfrac{1-L}{L}$.
Since $L$ is a $\textsc{beta}(a_0, b_0)$ variable, a change of variables gives that the density function of $\Psi$ is
\[
\frac{1}{\textrm{Beta}(a_0,b_0)}\,\frac{\sin^{b_0-1}(\psi) \cos^{a_0-1}(\psi)}{(\sin \psi+\cos \psi)^{a_0+b_0}}\,, \quad \mbox{for $\psi \in (0,\pi/2)$}\,.\]

In terms of this variable $\Psi$, we have
\[
\lim_{n \to \infty}\frac{r_n}{n}=\frac{1}{\sin \Psi+\cos\Psi}\,, \quad \mbox{almost surely}\,,\]
and
\[\lim_{n \to \infty} \psi_n=\Psi\,,  \quad \mbox{almost surely}\,.\]

In the basic case, when  $a_0=b_0=1$, the density 
 function of $\Psi$ is $1/(1+\sin(2 \psi))$, for $\psi \in (0,\pi/2)$, which approaches 1 (its supremum) as $\psi\to 0$ and as $\psi\to \pi/2$, and takes its minimum value, 1/2, at $\psi=\pi/4$.

\subsection{Polya's walk as mixture of $\alpha$-random walks}\label{section:polya walk as mixture}

In this section, we consider  the P\'{o}lya walk $(\mathbb{Y}_n)_{n\ge 0}$ and $\alpha$-random walks $(\mathbb{Z}_{\alpha,n})_{n\ge 0}$,  all of them starting  at some initial point $\mathbb{Y}_0=\mathbb{Z}_{\alpha,0}=(a_0,b_0)\in \mathbb{N}^2$.

Conditioning on $L=\alpha$, the sequence $(F_j)_{j\ge 1}$ consists of independent Bernoulli variables with parameter $\alpha$, and thus the distribution of P\'{o}lya's walk $(\mathbb{Y}_n)_{n\ge 1}$, \textit{conditioned on the limit $L$ taking the value $\alpha\in (0,1)$},  coincides with the distribution of the $\alpha$-random walk~$(\mathbb{Z}_{\alpha,n})_{n\ge 1}$:
\begin{equation}
\label{eq:contioning on limit}
(\mathbb{Y}_1, \ldots, \mathbb{Y}_{n} \,|\, L=\alpha)\overset{\rm d}{=} (\mathbb{Z}_{\alpha,1}, \ldots, \mathbb{Z}_{\alpha,{n}})\, , \quad \mbox{for each ${n}\ge 1$}\, . \end{equation}

For a fixed time $n \ge 1$, equation  \eqref{eq:contioning on limit} means, in particular,  that
\[
\P( \mathbb{Y}_n=\mathbb{Y}_0+(k,n-k)\,|\, L=\alpha) =\binom{n}{k} \alpha^k (1-\alpha)^{n-k}=\P(\mathbb{Z}_{\alpha,n}=\mathbb{Z}_{\alpha,0}+(k,n-k))\, , \]
for every $k,n$ such that  $0 \le k \le n$ and all $\alpha \in (0,1)$,
and that
\begin{equation}\label{eq:mixture of random walks one step}
\begin{aligned}\P( \mathbb{Y}_n=\mathbb{Y}_0+(k,n-k))&=\int_0^1 \P(\mathbb{Z}_{\alpha,n}=\mathbb{Z}_{\alpha,0}+(k,n-k)) \,f_{a_0,b_0}(\alpha) \,d\alpha
\\&=\int_0^1 \P(\textsc{bin}(n,\alpha)=k)f_{a_0,b_0}(\alpha) \,d\alpha \,,\end{aligned}\end{equation}
for every $k,n$ such that $0 \le k \le n$.

In other terms, \textit{P\'{o}lya's walk $ (\mathbb{Y}_n )_{n \ge 0}$ starting from the initial position $(a_0,b_0)$ is a mixture of\/ $\alpha$-random walks $ (\mathbb{Z}_{\alpha,n} )_{n \ge 0}$, all starting at $(a_0,b_0)$, where the mixture parameter $\alpha \in (0,1)$ follows a $\textsc{Beta}(a_0,b_0)$ probability distribution.}

In the basic case when the initial position is $(a_0,b_0)=(1,1)$, the density $f_{1,1}(\alpha)$ is identically 1 in $[0,1]$ and, for each $n \ge 1$ and $0 \le k \le n$,
\begin{equation}
\label{eq:mixture of random walks one step from 11}\P( \mathbb{Y}_n=(1,1)+(k,n-k)) =\int_0^1 \binom{n}{k} \alpha^k (1-\alpha)^{n-k} \,d\alpha
=\frac{1}{n+1}\,\cdot \end{equation}
Thus, in that case,  the possible positions at time $n$ of P\'{o}lya's walk starting from $(1,1)$ are equally likely.

%

\smallskip
Recall, from \eqref{eq:def QN} and \eqref{eq:def SNa}, the random variables $Q_N$ and $S_{\alpha, N}$ registering proportion of visible times up to time $N$ for  P\'{o}lya's walk and the $\alpha$-random walk:
\[Q_N=\frac{1}{N} \sum_{n=1}^N I(\mathbb{Y}_n) \quad \mbox{and}\quad S_{\alpha,N}=\frac{1}{N} \sum_{n=1}^N I(\mathbb{Z}_{\alpha,n})\,, \quad \mbox{for $\alpha\in (0,1)$ and $N\ge 1$}\, .\]
\begin{lemma}\label{lemma:media y var de QN}
For each $N \ge 1$, we have
\begin{align}
\label{eq:expectation of proportion of visible at time n}
\E(Q_N)&=\int_0^1 \E\big(S_{\alpha,N}\big)\, f_{a_0,b_0}(\alpha)\, d \alpha\, ,
\\
\label{eq:variance of proportion of visible at time n, con H}
\V(Q_N)&=\int_0^1 \V(S_{\alpha,N})\, f_{a_0,b_0}(\alpha)\, d \alpha+\V(H_N)\,,
\end{align}
where $H_N$ is a variable defined in the probability space $(0,1)$ with probability density $f_{a_0,b_0}(\alpha)$ and with values
$H_N(\alpha)=\E(S_{\alpha,N})$ for each $\alpha \in (0,1)$.
\end{lemma}

Equations \eqref{eq:expectation of proportion of visible at time n} and 
\eqref{eq:variance of proportion of visible at time n, con H} will allow us to translate appropriate  estimates of the means and variances of the average time $S_{\alpha,N}$ of the random walks $\mathbb{Z}_{\alpha,n}$ into estimates of the mean and  variance of the average time $Q_N$ of P\'{o}lya's walk $\mathbb{Y}_{n}$.

\begin{proof}
As a consequence of \eqref{eq:mixture of random walks one step}, for the P\'{o}lya walk starting at $(a_0,b_0)$ we have
\begin{equation}
\label{eq:expectation of visible at time n}
\E(I(\mathbb{Y}_n))=\int_0^1 \E(I(\mathbb{Z}_{\alpha,n}))\, f_{a_0,b_0}(\alpha)\, d \alpha\, ,
\end{equation}
for $n \ge 1$. This is so because
\begin{align*}
\E(I(\mathbb{Y}_n))
&
=\sum_{k=0}^n \P(\mathbb{Y}_n=(a_0,b_0)+(k,n-k)) \cdot  I\big((a_0,b_0)+(k,n-k)\big)\\
&=\int_0^1 \Big[\sum_{k=0}^n\P(\mathbb{Z}_{\alpha,n}=(a_0,b_0)+(k,n-k))  \cdot  I((a_0,b_0)+(k,n-k))\Big]\,f_{a_0,b_0}(\alpha) \,d\alpha\\
&=\int_0^1 \E(I(\mathbb{Z}_{\alpha,n}) )\, f_{a_0,b_0}(\alpha)\, d \alpha\, .
\end{align*}
Equation \eqref{eq:expectation of visible at time n} gives immediately,  for walks starting at $(a_0,b_0)$, that
\begin{equation*}
\E(Q_N)=\int_0^1 \E(S_{\alpha,N})\, f_{a_0,b_0}(\alpha)\, d \alpha\, ,\quad \mbox{for each $N \ge 1$}\,,
\end{equation*}
which is the first claim of the lemma.

Now, for times $ n$ and $n+m$, where $n,m\ge 1$, and positions $0\le k \le n$ and $0 \le q \le m$, one has that for all $\alpha \in (0,1)$,
\begin{align*}
\P( \mathbb{Y}_n &=(a_0,b_0)+(k,n-k), \mathbb{Y}_{n+m}=\mathbb{Y}_n+(q,m-q)\,|\, L=\alpha)
\\&
\qquad =\binom{n}{k} \alpha^k (1-\alpha)^{n-k}
\binom{m}{q} \alpha^q (1-\alpha)^{m-q}
\\
&\qquad =\P(\mathbb{Z}_{\alpha,n}=(a_0,b_0)+(k,n-k), \mathbb{Z}_{\alpha,n+m}=\mathbb{Z}_{\alpha,n}+(q,m-q))\, ,
\end{align*}
and that
\begin{equation}\label{eq:mixture of random walks two steps}
\begin{aligned}
&\P( \mathbb{Y}_n=(a_0,b_0)+(k,n-k),\, \mathbb{Y}_{n+m}=\mathbb{Y}_n+(q,m-q))
\\&=\int_0^1 \P\big(\mathbb{Z}_{\alpha,n}=(a_0,b_0)+(k,n-k),\, \mathbb{Z}_{\alpha,n+m}=\mathbb{Z}_{\alpha,n}+(q,m-q)\big)\,f_{a_0,b_0}(\alpha) \,d\alpha
\\&=\int_0^1 \P(\textsc{bin}(n,\alpha)=k)\,\P(\textsc{bin}(m,\alpha)=q)\, f_{a_0,b_0}(\alpha) \,d\alpha \,.\end{aligned}
\end{equation}

As a consequence of \eqref{eq:mixture of random walks two steps}, and analogously as the derivation of \eqref{eq:expectation of visible at time n}, we deduce that for walks starting at $(a_0,b_0)$ we have, for $n, m \ge 1$,  that
\begin{equation}
\label{eq:expectation of visible at time n and time m}
\E(I(\mathbb{Y}_n)  \cdot  I(\mathbb{Y}_{n+m}))=\int_0^1 \E(I(\mathbb{Z}_{\alpha,n}) \cdot  I(\mathbb{Z}_{\alpha,n+m})\big)\, f_{a_0,b_0}(\alpha)\, d \alpha\, .
\end{equation}

This gives that
\begin{equation*}
\E(Q_N^2)=\int_0^1 \E(S_{\alpha,N}^2)\, f_{a_0,b_0}(\alpha)\, d \alpha\, ,\quad \mbox{for each $N \ge 1$}\,,
\end{equation*}
and, in particular, that, for each $N \ge 1$,
\begin{align*}
&\V(Q_N)=\E(Q_N^2)-\E(Q_N)^2
\\&=
\int_0^1 \V(S_{\alpha,N})\, f_{a_0,b_0}(\alpha)\, d \alpha
+
\Big[\int_0^1 \E(S_{\alpha,N})^2 f_{a_0,b_0}(\alpha)\, d \alpha- \Big(\int_0^1 \E(S_{\alpha,N}) f_{a_0,b_0}(\alpha)\, d \alpha\Big)^2\Big].
\end{align*}
The second line in this equation contains the variance of a variable $H_N$ defined in the probability space $(0,1)$ with probability density $f_{a_0,b_0}(\alpha)$ and with values
\begin{equation*}
H_N(\alpha)=\E(S_{\alpha,N})\, , \quad \mbox{for each $\alpha \in (0,1)$}\,,
\end{equation*}
as claimed.
\end{proof}

\section{Asymptotic visibility of P\'{o}lya's walks: proof of Theorem \ref{teor:asymptotic coprime proportion of polya walks}}\label{section:proof of main}

For the mean and variance of the average visible time $S_{\alpha, N}$ of the $\alpha$-random walk, we have the following estimate.
\begin{prop}\label{prop:bounds for mean and var of RW dependence on alpha}
For each $\alpha \in  (0,1)$, the $\alpha$-random walk $(\mathbb{Z}_{\alpha,n})_{n\ge 0}$ with any given initial position $(a_0,b_0)\in \mathbb{N}^2$ satisfies the estimates
\begin{align}
\label{eq:mean of average time random walk}\E(S_{\alpha,N})&=\frac{1}{\zeta(2)}+\frac{1}{\sqrt{\alpha(1-\alpha)}}\, O\Big( \frac{1} {N^{1/4}}\Big)\, , \quad \mbox{as $N \to \infty$}\,,\\ \intertext{and}
\label{eq:variance of average time random walk}\V(S_{\alpha,N})&=\frac{1}{\alpha(1-\alpha)}\, O\Big( \frac{1} {N^{1/4}}\Big)\,, \quad \mbox{as $N \to \infty$}\,.
\end{align}
\end{prop} The implied constants in the big-O above depend  on $(a_0,b_0)$, but not upon $\alpha\in (0,1)$.

Proposition \ref{prop:bounds for mean and var of RW dependence on alpha} appears in \cite{CFF} for the case $(a_0,b_0)=(0,0)$ (as Propositions 3.1 and 3.2 there), except for the error terms: here, the factor $N^{-1/4}$ will be enough, but we record the explicit dependence of the bounds on the probability $\alpha$. For our proof of Theorem~\ref{teor:asymptotic coprime proportion of polya walks}, this dependence in $\alpha$ is an essential component. Details of the proof of Proposition~\ref{prop:bounds for mean and var of RW dependence on alpha} are the content of Section \ref{section:estimates for random walk}.

Now we translate Proposition \ref{prop:bounds for mean and var of RW dependence on alpha}  about $\alpha$-random walks into a corresponding result for P\'{o}lya's walk and appeal to the second moment method (Proposition \ref{prop:Borel}) towards completing the proof of Theorem \ref{teor:asymptotic coprime proportion of polya walks}.

For the average visible time $Q_N$ of P\'{o}lya's walk, we have the following asymptotic estimates for its mean and variance.
\begin{cor}\label{cor:bounds for mean and var of PW} For the P\'{o}lya walk $(\mathbb{Y}_n)_{n\ge 0}$ with initial position $(a_0,b_0)$ such that $a_0,b_0\ge 2$ we have that, as $N \to \infty$,
\[
\E(Q_{N})=\frac{1}{\zeta(2)}+O\Big( \frac{1} {N^{1/4}}\Big)\quad \mbox{and}\quad
\V(Q_{N})=O\Big( \frac{1} {N^{1/4}}\Big)\,.
\]
\end{cor}
Note the restriction $a_0,b_0\ge 2$. 

Again, the implied constants in the big-O above depend on $(a_0,b_0)$, but not upon $\alpha\in (0,1)$.

\begin{proof}
[Proof of Corollary {\rm\ref{cor:bounds for mean and var of PW}}]
The result for $\E(Q_{N})$ follows by applying the estimate
\eqref{eq:mean of average time random walk} of Proposition \ref{prop:bounds for mean and var of RW dependence on alpha}
to the formula \eqref{eq:expectation of proportion of visible at time n} and integrating on $\alpha$ (recalling the finiteness of the integral in \eqref{eq:integral finita Beta con raices}).

Formula \eqref{eq:variance of proportion of visible at time n, con H}
for $\V(Q_N)$ has two terms.
The first one is handled by integrating the estimate \eqref{eq:variance of average time random walk}
of Proposition \ref{prop:bounds for mean and var of RW dependence on alpha} with respect to $\alpha$. We use here that, 
 because $a_0, b_0\ge 2$, the integral in \eqref{eq:integral finita Beta sin raices} is finite.

For the second summand, $\V(H_N)$, observe first that, since the mean minimizes square deviation, one has  that
 \[
 \V(H_N) = \int_0^1 \big(\E(S_{\alpha,N})-\E(H_N)\big)^2 f_{a_0,b_0}(\alpha)\, d \alpha\le \int_0^1 \big(\E(S_{\alpha,N})-1/\zeta(2)\big)^2 f_{a_0,b_0}(\alpha)\, d \alpha\, ,
 \]
and then integrate with respect to $\alpha$  the estimate \eqref{eq:mean of average time random walk} of Proposition \ref{prop:bounds for mean and var of RW dependence on alpha}, using again that $a_0, b_0\ge 2$.
\end{proof}

The next proposition registers the so called \textit{second moment method} (see, for instance, Section 2.3 in \cite{Du}, and Lemma 2.4 in \cite{CFF}).
\begin{prop}\label{prop:Borel}Let $(W_n)_{n\ge 1} $ be a sequence of uniformly bounded random variables in a certain probability space, and let $U_N$ be the average
\[U_N=\frac{1}{N}(W_1+\dots+W_N)\, , \quad \mbox{for each $N \ge 1$}\,.\]
If\/
$
\lim_{N\to\infty}\E(U_N)=\mu,
$
and if for some $B,\delta>0$,
$\V(U_N)\le B/N^{\delta}$ for each $N\ge 1$, then
\[
\lim_{N\to \infty}U_N=\mu\, , \quad \mbox{almost surely}\,.
\]
\end{prop}

\begin{proof}
[Proof of Theorem {\rm \ref{teor:asymptotic coprime proportion of polya walks}}]
The second moment method of Proposition  \ref{prop:Borel} combined with the estimates of Corollary \ref{cor:bounds for mean and var of PW} gives the result that $\lim_{N \to \infty} Q_N=1/\zeta(2)$, almost surely, at least in the case when the starting point $(a_0,b_0)$ of the walk satisfies $a_0, b_0\ge 2$.

The remaining case of Theorem \ref{teor:asymptotic coprime proportion of polya walks}, that is, when $a_0=1$ or $b_0=1$, follows from observing that P\'{o}lya's walk $\mathbb{Y}_n=(a_n,b_n)$ starting at $(a_0,1)$ has null probability of remaining \textit{always} at height 1. This is so because 
\[\P\Big(\bigcap_{n=0}^N\{b_n=1\}\Big)=\frac{a_0}{a_0+N}\, ,\]
which tends to 0 as $N \to \infty$.

In particular, for P\'{o}lya's walk starting at $(1,1)$, both the probability that the first~$n$ moves are upwards and the probability that the first $n$ moves are rightwards is ${1}/(n+1)$, for $n \ge 1$. The walk enters the region $\{a \ge2, b\ge 2\}\subset \N^2$ with probability one, although the average time it takes  to do so is infinite. In fact, $\lim_{\to \infty} a_n=+\infty$ and $\lim_{\to \infty} b_n=+\infty$, almost surely.

This completes the proof of Theorem \ref{teor:asymptotic coprime proportion of polya walks}.
\end{proof}

\begin{remark}\label{remark:media polya en caso 1,1} For P\'{o}lya's walk starting at $(a_0,b_0)=(1,1)$, we can give a closed formula for $\E(Q_N)$. By appealing to equation \eqref{eq:mixture of random walks one step from 11} we have that
\[
\begin{aligned}\E(I(\mathbb{Y}_n))&=\frac{1}{n+1} \sum_{j=0}^n \delta_1\big(\gcd(j+1, n+1-j)\big)
\\
&=\frac{1}{n+1} \sum_{j=0}^n \delta_1\big(\gcd(j+1, n+2)\big)=\frac{\phi(n+2)}{n+1}\, , \quad \mbox{for each $n \ge 0$}\, ,
\end{aligned}
\]
where $\phi$ denotes, as usual, Euler's totient function.
And thus,
\[\E(Q_N)=\frac{1}{N}\sum_{n=1}^N \frac{\phi(n+2)}{n+1}\, , \quad \mbox{for each $N \ge 1$}\, .\]
This last expression is easily seen to converge to $1/\zeta(2)$ as $N\to \infty$, by comparing it with
\[\frac{1}{N} \sum_{n=1}^N \frac{\phi(n)}{n}=\frac{1}{N}\sum_{d=1}^N \frac{\mu(d)}{d} \Big\lfloor \frac{N}{d}\Big\rfloor
\xrightarrow{N\to\infty} \sum_{d =1}^\infty  \frac{\mu(d)}{d^2}=\frac{1}{\zeta(2)}\,
.
\]
Here we have used \eqref{eq:Mobius y Euler} and \eqref{eq:sum mu/d2entera}.
\end{remark}

\subsection{Estimates for $\alpha$-random walks: proof of Proposition \ref{prop:bounds for mean and var of RW dependence on alpha}}\label{section:estimates for random walk}

The aim of this section is to prove the estimates of Proposition \ref{prop:bounds for mean and var of RW dependence on alpha} about the $\alpha$-random walks $\mathbb{Z}_{\alpha,n}$.

The proof follows much along the lines of those of Propositions 3.1 and 3.2 in \cite{CFF}.

We first need a bound for binomial probabilities such as the following: for any $\alpha\in(0,1)$, $N\ge 1$ and $0\le k\le N$,
\[
\P(\textsc{bin}(N,\alpha)=k)\le C\, \frac{1}{\sqrt{N\alpha(1-\alpha)}}\,,
\]
for a certain constant $C>0$. This estimate could be obtained by appealing to the local central limit theorem
(see, for instance, Theorem 3.5.2 in \cite{Du}), or by combining the unimodality of the binomial probabilities with Stirling's approximation. For a precise value of $C$, we register the following result.

\begin{lemma}
For any $\alpha\in(0,1)$, $N\ge 1$ and $0\le k\le N$, we have
\begin{equation}
\label{eq:cota uniforme binomial}
\P(\textsc{bin}(N,\alpha)=k)\le \frac{\pi}{2}\, \frac{1}{\sqrt{2\pi N\alpha(1-\alpha)}}\,\cdot
\end{equation}
\end{lemma}

\begin{proof}
Write $\alpha=t/(1+t)$, so that $t=\alpha/(1-\alpha)$ and $\alpha(1-\alpha)=t/(1+t)^2$.

We start with the identity
\[
(\dag)\qquad \P(\textsc{bin}(N,\alpha)=k)= \frac{1}{2\pi} \int_{-\pi}^\pi \Big(\frac{1+te^{i\theta}}{1+t}\Big)^N\, e^{-ik\theta}\, d\theta,
\]
that can be verified by expanding, binomial theorem, the factor within the integral. Notice also that $\big(\frac{1+te^{i\theta}}{1+t}\big)^N$ is the characteristic function $\varphi_X(\theta)$ of a random variable $X\sim\textsc{bin}(N,\alpha)$.

Now we bound $(\dag)$ as follows:
\begin{align*}
&\P(\textsc{bin}(N,\alpha)=k)\le \frac{1}{2\pi} \int_{-\pi}^\pi \Big|\frac{1+te^{i\theta}}{1+t}\Big|^N\, d\theta
=\frac{1}{2\pi} \int_{-\pi}^\pi \Big(\Big|\frac{1+te^{i\theta}}{1+t}\Big|^2\Big)^{N/2}\, d\theta
\\
&\qquad =\frac{1}{2\pi} \int_{-\pi}^\pi \Big(1+\frac{2t}{(1+t)^2}\, (\cos \theta-1)\Big)^{N/2}\, d\theta
\stackrel{(\star)}{\le} \frac{1}{2\pi} \int_{-\pi}^\pi \exp\Big(\frac{Nt}{(1+t)^2}\, (\cos \theta-1)\Big)\, d\theta
\\
&\qquad \stackrel{(\star\star)}{\le}  \frac{1}{2\pi} \int_{-\pi}^\pi \exp\Big(-\frac{2}{\pi^2} \frac{Nt}{(1+t)^2}\, \theta^2\Big)\, d\theta
=\frac{1}{2\pi} \int_{-\pi}^\pi \exp\Big(-\frac{2N}{\pi^2} \alpha(1-\alpha)\, \theta^2\Big)\, d\theta
\\
&\qquad \le \frac{1}{2\pi} \int_{-\infty}^\infty \exp\Big(-\frac{2N}{\pi^2} \alpha(1-\alpha)\, \theta^2\Big)\, d\theta
=\frac{\pi}{2}\, \frac{1}{\sqrt{2\pi N\alpha(1-\alpha)}}\,\cdot
\end{align*}
In $(\star)$, we have used $1+x\le e^x$ for $x\in\mathbb{R}$, and the bound $\cos \theta-1\le -2\theta^2/\pi^2$, valid for $|\theta|\le \pi$, was employed for $(\star\star)$.
\end{proof}

By the way, inequality \eqref{eq:cota uniforme binomial} is not true substituting $\pi/2$ by 1.

\smallskip

Next, we need an estimate for sums of binomial probabilities restricted to indices in a certain residue class.
\begin{lemma}\label{lemma:binomial probabilities residue class} There is an absolute constant $C>0$ such that for any $\alpha\in(0,1)$ and for integers $n\ge 1$, $d\ge 1$, and $r\in\Z$, there holds
\begin{equation}\label{eq:binomial probabilities residue class}
\Big|\sum_{\substack{0\le l\le n;\\l\equiv r{ \ \text{\rm mod}\, } d}}\binom nl \alpha^l(1-\alpha)^{n-l}-\frac1d\Big|\le \frac{C}{\sqrt{\alpha(1-\alpha)}\, \sqrt{n}}.
\end{equation}
\end{lemma}

We may restrict $r$ to $r \in \{0,1, \ldots, d-1\}$ or, for that matter, to any complete set of residues mod $d$, with no loss of generality.

For $r \in \{0,1, \ldots, d-1\}$,  we denote by $A_r=\{l\in \{0,\ldots, n\}: l\equiv r{ \ \text{\rm mod}\, } d\}$. There are $d$ of these \textit{classes}. Lemma \ref{lemma:binomial probabilities residue class} means that
$
\P (\textsc{bin}(n, \alpha)\in A_r )$ is approximately $1/d$, quite uniformly.

This lemma is stated  in \cite{CFF} with an unspecified constant $C_\alpha$ instead of $C/\sqrt{\alpha(1-\alpha)}$.

\begin{proof} 
We may assume $r\in\{0,\ldots, d-1\}$. If we assume further that $d\le n$, then the proof is as that of Lemma 2.1 in \cite{CFF}.

To remove the assumption $d \le n$, observe that for $d>n$ and any $r\in \Z$,  there are at most two values of $l$ in $\{0, \ldots, n\}$, congruent to $r$ mod $d$. (Actually,  at most 1 such value, except in the extreme case $r=0$, and $d=n$, where there are 2.) Thus the sum of probabilities in \eqref{eq:binomial probabilities residue class} is, by
\eqref{eq:cota uniforme binomial}, at most
$C  /{\sqrt{\alpha(1-\alpha)n}}$, where $C$ is an absolute constant
while
\[
\frac{1}{d}\le \frac{1}{n}\le \frac{2}{\sqrt{n}}\le \frac{1}{\sqrt{\alpha(1-\alpha)n}}\,\cdot\qedhere
\]
	\end{proof}

From Lemma \ref{lemma:binomial probabilities residue class}, we deduce the following estimate of some further restricted binomial sums.
This estimate is the key ingredient of our  proof of Proposition \ref{prop:bounds for mean and var of RW dependence on alpha}
and thus of Theorem \ref{teor:asymptotic coprime proportion of polya walks}.
\begin{lemma}
\label{lema:imp}
Let $\alpha\in(0,1)$. For any  integers  $s,t\ge 0$, we have that, as $M\to\infty$,
\begin{align*}
&\sum_{\substack{0\le l\le M\\ \gcd (l+s,M+t)=1}}\binom Ml\alpha^l(1-\alpha)^{M-l}
=\sum_{d\mid M+t}\frac{\mu(d)}d+\frac{1}{\sqrt{\alpha(1-\alpha)}}\, O\Big(\frac{\tau(M+t)}{\sqrt{M}}\Big)\,.
\end{align*}
\end{lemma}

The implied constant of the big-O of Lemma \ref{lema:imp} is absolute.

The proof of Lemma~\ref{lema:imp} is similar to that of Lemma 2.2 in \cite{CFF}. It results from combining  Lemma \ref{lemma:binomial probabilities residue class} and \eqref{eq:representation of indicator of gcd} with $k=1$ and rearranging terms. We include it for completeness.

\begin{proof}
We have that \[\delta_1(\gcd (l+s,M+t))=\sum_{d\mid \gcd (l+s,M+t)}\mu(d).\]
Thus,
\begin{align*}
&\sum_{0\le l\le M}\binom Ml\alpha^l(1-\alpha)^{n-l}\delta_1(\gcd (l+s,M+t))
=\sum_{d\mid M+t}\mu(d)
\sum_{\substack{0\le l\le M\\l\equiv -s \ \text{(mod $d$)}}}\binom Ml\alpha^l(1-\alpha)^{M-l}
\\
&\quad =\sum_{d\mid M+t}\mu(d)\Big(\frac 1d+\frac{1}{\sqrt{\alpha(1-\alpha)}}\,O\Big(\frac 1{\sqrt{ M}}\Big)  \Big)
=\sum_{d\mid M+t}\frac{\mu(d)}d+ \frac{1}{\sqrt{\alpha(1-\alpha)}}\,O\Big(\frac{\tau(M+t)}{\sqrt{M}}\Big)
\end{align*}
as $M\to\infty$. Lemma \ref{lemma:binomial probabilities residue class} justifies the second equality sign.\end{proof}

\smallskip
We write down now estimates for the means $\E\big(I(\mathbb{Z}_{\alpha,n})\big)$ and for $\E\big(I(\mathbb{Z}_{\alpha,n})I(\mathbb{Z}_{\alpha,m})\big)$ for the $\alpha$-random walk $\mathbb{Z}_{\alpha,n}$ starting at $(a_0,b_0)$.

\begin{lemma} \label{lemma:mean variance Zan}We have
\begin{equation}
\label{eq:mean of Xi}
\E\big(I(\mathbb{Z}_{\alpha,n})\big)
=\sum_{d\mid n+a_0+b_0}\frac{\mu(d)}d+\frac{1}{\sqrt{\alpha(1-\alpha)}}\, O\Big (\frac{1}{n^{1/4}}\Big )\, , \quad\text{as $n\to\infty$,}
\end{equation}
and, for $n<m$,
\begin{align}\nonumber
\E\big(I(\mathbb{Z}_{\alpha,n})\cdot I(\mathbb{Z}_{\alpha,m})\big)
&=\sum_{d\mid n+a_0+b_0}\frac{\mu(d)}d\sum_{d\mid m+a_0+b_0}\frac{\mu(d)}d
\\ \label{E(XiXj)}
&\quad +\frac{1}{\sqrt{\alpha(1-\alpha)}}\,O\Big (\frac{1}{n^{1/4}}\Big )+\frac{1}{\alpha(1-\alpha)}\,O\Big (\frac{m^{1/4}}{\sqrt{m-n}}\Big )
\end{align}
as $n\to\infty$.
\end{lemma}

The implied constants of the different big-O above do not depend upon $\alpha$.

The proof of Lemma \ref{lemma:mean variance Zan} is essentially the same as the proof of Lemma 2.6 in \cite{CFF}, except that we keep the explicit dependence on $\alpha$ in the error terms, that the initial point $(a_0,b_0)$ has to be taken into account, and that we have simplified the big-O's using, see \eqref{eq:order of tau}, that  
$\tau(n)=O(n^{1/4})$ as $n\to\infty$. We include it for completeness.
%
%

\begin{proof}[Proof of Lemma {\rm \ref{lemma:mean variance Zan}}]
Let $\mathbb{Z}_{\alpha,n}=(a_n,b_n)$ be the position at time $n$ of  an $\alpha$-random walk starting at $(a_0,b_0)$. Observe that $a_n+b_n=a_0+b_0+n$, so we can write $\mathbb{Z}_{\alpha,n}=(a_0,b_0)+(l, n-l)$ for some $l=0,1,\dots,n$. The probability that $\mathbb{Z}_{\alpha,n}$ equals $(a_0+l,b_0+n-l)$ is
\[
\binom nl\alpha^l(1-\alpha)^{n-l}.
\]
Notice that $\gcd(a_0+l,b_0+n-l)=\gcd(a_0+l,a_0+b_0+n)$. Thus,
\[
\E\big(I(\mathbb{Z}_{\alpha,n})\big)=\sum_{\substack{0\le l\le n\\ \gcd(a_0+l,a_0+b_0+n)=1}}\binom nl \, \alpha^l\,(1-\alpha)^{n-l}\, ,
\]
and Lemma \ref{lema:imp}, with $M=n$, $s=a_0$ and $t=a_0+b_0$, combined with the estimate \eqref{eq:order of tau} gives \eqref{eq:mean of Xi}.

\smallskip

Take now two positions, say $\mathbb{Z}_{\alpha,n}$ and $\mathbb{Z}_{\alpha,m}$, with $n<m$, of the $\alpha$-random walk starting at the point $(a_0,b_0)$. The coordinates of these two positions will be $\mathbb{Z}_{\alpha,n}=(a_0,b_0)+(l,n-l)$ and $\mathbb{Z}_{\alpha,n}=(a_0,b_0)+(l,n-l)+(r,m-n-r)$ for some $0\le l\le n$ and $0\le r\le m-n$, with probability
\[
\binom nl\alpha^l(1-\alpha)^{n-l}\binom{m-n}{r}\alpha^{r}(1-\alpha)^{m-n-r}.
\]
Note that $\gcd(a_0+l+r,b_0+m-l-r)=\gcd(a_0+l+r,a_0+b_0+m)$. Then,
\begin{align*}
&\E\big(I(\mathbb{Z}_{\alpha,n})I(\mathbb{Z}_{\alpha,m})\big)
\\
&=\sum_{\substack{0\le l\le n\\ \gcd(a_0+l,a_0+b_0+n)=1}}\binom nl\alpha^l(1-\alpha)^{n-l}
\sum_{\substack{0\le r\le m-n\\ \gcd(a_0+l+r,a_0+b_0+m)=1}}\binom{m-n}{r}\alpha^{r}(1-\alpha)^{m-n-r} .
\end{align*}
Now, using Lemma \ref{lema:imp} with $M=m-n$, $s=a_0+l$ and $t=a_0+b_0+n$, and so that $M+t=a_0+b_0+m$, in the inner sum, and then with $M=n$, $s=a_0$ and $t=a_0+b_0$ in the first sum,
combined with the estimate \eqref{eq:order of tau}, we get
\begin{align*}\nonumber
&\E\big(I(\mathbb{Z}_{\alpha,n})\cdot I(\mathbb{Z}_{\alpha,m})\big)
\\
&=\Big( \sum_{d\mid n+a_0+b_0}\!\frac{\mu(d)}d+  \frac{1}{\sqrt{\alpha(1-\alpha)}}\, O\Big(\frac{1}{n^{1/4}}\Big)    \Big)\! \Big( \sum_{d\mid m+a_0+b_0}\!\frac{\mu(d)}d+ \frac{1}{\sqrt{\alpha(1-\alpha)}}\,O\Big (\frac{m^{1/4}}{\sqrt{m-n}}\Big ) \Big)
\\
&=\sum_{d\mid n+a_0+b_0}\frac{\mu(d)}d\sum_{d\mid m+a_0+b_0}\frac{\mu(d)}d+ \frac{1}{\sqrt{\alpha(1-\alpha)}}\,O\Big(\frac{1}{n^{1/4}}\Big)
\\
&\quad +\frac{1}{\sqrt{\alpha(1-\alpha)}}\, O\Big (\frac{m^{1/4}}{\sqrt{m-n}}\Big )
+ \frac{1}{{\alpha(1-\alpha)}}\,O\Big (\frac{m^{1/4}}{n^{1/4}\, \sqrt{(m-n)}}\Big ) 
.
\end{align*}
Here, we have used \eqref{eq:Mobius y Euler}, and the fact that $\phi(k)/k< 1$. 
 Finally, observe that the last two terms can be written together as
\[
\frac{1}{{\alpha(1-\alpha)}}\, O\Big (\frac{m^{1/4}}{\sqrt{m-n}}\Big ).\qedhere
\]
\end{proof}

\begin{proof}[Proof of Proposition {\rm\ref{prop:bounds for mean and var of RW dependence on alpha}}]
Now, Proposition \ref{prop:bounds for mean and var of RW dependence on alpha} follows from Lemma~\ref{lemma:mean variance Zan} with an argument much akin to that proving Propositions 3.1 and 3.2 in \cite{CFF}.
\end{proof}

\begin{proof}[Proof of Theorem {\rm\ref{teor:basic asymptotic of CFF}}]
Theorem \ref{teor:basic asymptotic of CFF}, with no restriction on the departing point $(a_0,b_0)$, follows from Propositions \ref{prop:bounds for mean and var of RW dependence on alpha} and~\ref{prop:Borel}.
\end{proof}

\section{Asymptotic $k$-visibility of random walks and P\'{o}lya's walks}\label{section:gcd=k}

For $\alpha \in (0,1)$ and $k \ge 1$, and a given initial position $(a_0,b_0)$, we now consider the average time $S^{(k)}_{\alpha, N}$, up to time $N$, that the $\alpha$-random walk $\mathbb{Z}_{\alpha, n}$ is $k$-visible from the origin:
\[S^{(k)}_{\alpha, N}=\frac{1}{N} \sum_{n=1}^{N} I_k(\mathbb{Z}_{\alpha, n})\, , \quad \mbox{for $N \ge 1$}\, .\]

For $k \ge 1$, and a given initial position $(a_0,b_0)$, we consider also the average time that P\'{o}lya's walk $\mathbb{Y}_n$ is $k$-visible from the origin:
\[
Q^{(k)}_{N}=\frac{1}{N} \sum_{n=1}^{N} I_k(\mathbb{Y}_{n})\, , \quad \mbox{for $N \ge 1$}\, .
\]

The analysis for visibility ($k=1$) which has been carried out in the previous sections may be extended to $k$-visibility.  The key new ingredient is the following extension of Lemma \ref{lema:imp}.

\begin{lemma}
\label{lema:imp for k}
Let $\alpha\in(0,1)$ and $k\ge 1$. For any  integers $s,t\ge 0$, we have that, as $M\to\infty$,
\begin{align*}
\sum_{\substack{0\le l\le M\\ \gcd (l+s,M+t)=k}}\binom Ml\alpha^l(1-\alpha)^{M-l}
&=\frac{g_k(M+t)}{k}\sum_{kd\mid M+t}\frac{\mu(d)}d
\\
&\quad +\frac{g_k(M+t)}{\sqrt{\alpha(1-\alpha)}}\, O\Big(\frac{\tau((M+t)/k)}{\sqrt{M}}\Big)\, .
\end{align*}
\end{lemma}

Recall, from Section \ref{subseccion:number-theoretical}, $g_k$ denotes the arithmetic function given by $g_k(n)=1$ if~$n$ is a multiple of $k$, and $g_k(n)=0$ otherwise.
The proof of Lemma \ref{lema:imp for k}, like that of Lemma~\ref{lema:imp}, results from combining \eqref{eq:representation of indicator of gcd}, now for general $k \ge 1$, and Lemma \ref{lemma:binomial probabilities residue class}.

For the proportion  of $k$-visibility $S^{(k)}_{\alpha,N}$, we have the following extension of Proposition~\ref{prop:bounds for mean and var of RW dependence on alpha}.

\begin{prop}\label{prop:bounds for mean and var of RW dependence on alpha with k} For each $\alpha \in  (0,1)$ and $k \ge 1$, the $\alpha$-random walk $(\mathbb{Z}_{\alpha,n})_{n\ge 0}$ with any given initial position $(a_0,b_0)$ satisfies the estimates
\begin{align*}
\E(S^{(k)}_{\alpha,N})&=\frac{1}{k^2\zeta(2)}+\frac{1}{\sqrt{\alpha(1-\alpha)}}\,O\Big( \frac{1} {N^{1/4}}\Big)\, , \quad \mbox{as $N \to \infty$}\,,\\ \intertext{and}
\V(S^{(k)}_{\alpha,N})&=\frac{1}{\alpha(1-\alpha)}\, O\Big( \frac{1} {N^{1/4}}\Big)\,,\quad \mbox{as $N \to \infty$}\,.
\end{align*}
\end{prop} The implied constants in the big-O's above depend on  $(a_0,b_0)$ and on $k$, but not upon $\alpha\in (0,1)$.

\begin{proof} We content ourselves with explaining the argument in the case when the starting point is $(a_0,b_0)=(1,1)$ and just for the mean $\E(S^{(k)}_{\alpha,N})$. This argument exhibits the only differences with the case $k=1$.

Lemma \ref{lema:imp for k} with $t=s=0$ and $M=n$ gives
\[
\E\big(I_k(\mathbb{Z}_{\alpha, n})\big)=\frac{g_k(n)}{k}\sum_{kd\mid n}\frac{\mu(d)}{d}+\frac{g_k(n)}{\sqrt{\alpha(1-\alpha)}}\, O\Big(\frac{\tau(n/k)}{\sqrt{n}}\Big)\,,\quad\text{as $n\to\infty$}.
\]
This expectation is non zero only if $n$ is a multiple of $k$.

Let $N\ge 1$ and let $m=\lfloor N/k\rfloor$, so that $mk \le N <(m+1) k$.

Then, with big-O's depending on $k$, we have
\[\begin{aligned} N\, \E(S^{(k)}_{\alpha,N})&=\sum_{j=1}^m \E\big(I_k(\mathbb{Z}_{\alpha, kj})\big)=\frac{1}{k}\sum_{j=1}^m\sum_{d\mid j}\frac{\mu(d)}{d}+\frac{1}{\sqrt{\alpha(1-\alpha)}}\, O\Big(\sum_{j=1}^m\frac{\tau(j)}{\sqrt{j}}\Big)
\\
&=\frac{1}{k}\sum_{d=1}^m \frac{\mu(d)}{d} \Big\lfloor \frac{m}{d}\Big\rfloor +\frac{1}{\sqrt{\alpha(1-\alpha)}}O\big({m^{3/4}}\big)
=\frac{1}{k}\frac{m}{\zeta(2)} +\frac{1}{\sqrt{\alpha(1-\alpha)}}\,O\big({m^{3/4}}\big)\\
&=\frac{1}{k^2}\frac{N}{\zeta(2)}+\frac{1}{\sqrt{\alpha(1-\alpha)}}\,O\big({N^{3/4}}\big)\, , \quad \mbox{as $N \to \infty$}\,,
\end{aligned}
\]
where we have used \eqref{eq:sum mu/d2entera}.
And thus,
\[\E(S^{(k)}_{\alpha,N})=\frac{1}{k^2\zeta(2)}+\frac{1}{\sqrt{\alpha(1-\alpha)}}\,O\Big(\frac{1}{N^{1/4}}\Big)\, , \quad \mbox{as $N \to \infty$}\,.\qedhere\]
\end{proof}

For the proportion of  $k$-visibility time, $Q^{(k)}_N$, of P\'{o}lya's walk we deduce, as a corollary of Proposition \ref{prop:bounds for mean and var of RW dependence on alpha with k}, that:
\begin{cor}\label{cor:bounds for mean and var of PW with k} For P\'{o}lya's walk $(\mathbb{Y}_n)_{n\ge 0}$ with any given initial position $(a_0,b_0)$ such that $a_0,b_0\ge 2$ we have that, as $N \to \infty$,
\[
\E(Q^{(k)}_{N})=\frac{1}{k^2\zeta(2)}+O\Big( \frac{1} {N^{1/4}}\Big)\quad\text{and}\quad
\V(Q^{(k)}_{N})=O\Big( \frac{1} {N^{1/4}}\Big)\,.
\]
\end{cor}
The implied constants in the big-O above depends on $(a_0,b_0)$ and $k$.

\smallskip
To prove Theorem \ref{teor:asymptotic gcd=k proportion of polya walks}, we deduce, exactly as in the case $k=1$,  that  for fixed $k \ge 1$,
\[
\lim_{N \to \infty} \frac{1}{N} \sum_{n=1}^N I_k\big(\mathbb{Y}_n\big)=\frac{1}{k^2\zeta(2)}\, ,\quad \mbox{almost surely}\,,
\]
and, thus, that
\[\mbox{almost surely}, \quad \lim_{N \to \infty} \frac{1}{N} \sum_{n=1}^N I_k\big(\mathbb{Y}_n\big)=\frac{1}{k^2\zeta(2)}\, , \quad \mbox{for each $k \ge 1$}\, .\]
which is the assertion of Theorem \ref{teor:asymptotic gcd=k proportion of polya walks}.

Likewise, for $\alpha$-random walks we have, as announced after  Theorem  \ref{teor:basic asymptotic of CFF}, that
 \begin{thm}\label{teor:basic asymptotic of CFF with k} For any $\alpha \in (0,1)$,   the $\alpha$-random walk $\mathbb{Z}_{\alpha,n}$ starting at any given initial position $(a_0,b_0)$ satisfies
 \[
 \mbox{almost surely} \quad \lim_{N \to \infty} \frac{1}{N} \sum_{n=1}^N I_k(\mathbb{Z}_{\alpha, n})=\frac{1}{k^2\zeta(2)}\, , \quad \mbox{for all $k \ge 1$}\, .\]
 \end{thm}

\section{Changing the step of the walk}\label{section:Changing the steps of the walk}

We fix now an integer step $c\ge 1$. We now consider P\'{o}lya's walk $(\mathbb{Y}^c_n)_{n \ge 0}$,  starting from $\mathbb{Y}^c_0=(a_0, b_0)$ with $(a_0,b_0)\in \N^2$,  when the steps of the walk are of size $c$. This means that for  each $n\ge 0$, the jump  $\mathbb{Y}^c_{n+1}-\mathbb{Y}^c_n$ can take only two values, $(c,0)$ and $ (0,c)$: the walk either  moves $c$ units to the right or $c$ units up.
 The total linear  distance travelled from $\mathbb{Y}^c_0$ to $\mathbb{Y}^c_n$ is $nc$.

\subsection{Exchangeability and mixtures}

For each $n \ge 0$, \textit{given} the position $\mathbb{Y}^c_n=(a_n,b_n)$ of the walk at time $n$, the conditional probabilities of the only two admissible  jumps are
\[
\P\big(\mathbb{Y}^c_{n+1}-\mathbb{Y}^c_n=(c,0)\big)=\frac{a_n}{a_n+b_n}\quad \text{and}\quad
 \P\big(\mathbb{Y}^c_{n+1}-\mathbb{Y}^c_n=(0,c)\big)=\frac{b_n}{a_n+b_n}\,\cdot
 \]

This walk corresponds to a P\'{o}lya's urn process where at each stage, $c$ balls are added to the urn of the color  of the observed/drawn  ball. We always have at time $n$ that $a_n+b_n=a_0+b_0+nc$.

For each $\alpha \in (0,1)$, we also consider the $\alpha$-random walk $(\mathbb{Z}^c_{\alpha, n})_{n\ge 0}$ starting from $(a_0,b_0)$, with right-step $(c,0)$ with probability $\alpha$ and up-step $(0,c)$ with probability \mbox{$(1-\alpha)$}.

\smallskip

As we have discussed in the case $c=1$, see Section~\ref{section:exchange polya urn walk}, the sequence of Bernoulli variables $F^c_n$ registering whether the $n$-th step is to the right, $F^c_n=1$, or up, $F^n_n=0$, is exchangeable and in fact, if $(x_1, \ldots, x_n)$ is a list extracted from $\{0,1\}$ and if $t_n=\sum_{j=1}^n x_j$, then
\[
\begin{aligned}
\P(F^c_1=x_1, \ldots, F^c_n=x_n)&=\frac{\prod_{j=0}^{t_n-1} (a_0+jc)\, \prod_{j=0}^{n-t_n-1} (b_0+jc)}{\prod_{j=0}^{n-1} (a_0+b_0+jc)}\\[6pt]&=
\frac{\Gamma(a_0/c+t_n)}{\Gamma(a_0/c)}
\frac{\Gamma(b_0/c+n-t_n)}{\Gamma(b_0/c)}
\frac{\Gamma((a_0+b_0)/c)}{\Gamma((a_0+b_0)/n+n)}\\[6pt]&=\frac{\textrm{Beta}(a_0/c+t_n, b_0/c+n-t_n)}{\textrm{Beta}(a_0/c,b_0/c)}\,\cdot\end{aligned}\]
The de Finetti mixture measure is in this situation $d\nu(\alpha)=f_{a_0/c, b_0/c}(\alpha)\, d\alpha$, that is, a $\textsc{beta}(a_0/c, b_0/c)$ distribution.

\smallskip

The proportion $\frac{1}{n}\sum_{j=1}^n F^c_j$ converges almost surely to a variable $L^c$ which follows a $\textsc{beta}(a_0/c, b_0/c)$ distribution. Conditioning on $L^c=\alpha$, the sequence $(F_n^c)_{n\ge 1}$ consists of independent Bernoulli variables with parameter $\alpha$, and thus the distribution of P\'{o}lya's walk $(\mathbb{Y}^c_n)$ \textit{conditioned on the limit $L^c$ taking the value $\alpha\in (0,1)$}  coincides with the distribution of the $\alpha$-random walk $(\mathbb{Z}^c_{\alpha,n})$:
\begin{equation}
\label{eq:contioning on limit step c}
 (\mathbb{Y}^c_1, \ldots, \mathbb{Y}^c_N\,|\, L^c=\alpha )\overset{\rm d}{=}  (\mathbb{Z}^c_{\alpha,1}, \ldots, \mathbb{Z}^c_{\alpha,N} )\, , \quad \mbox{for each $N\ge 1$}\, . \end{equation}

For a fixed time $n \ge 1$, equation  \eqref{eq:contioning on limit step c} means, in particular,  that
\[
\P\big( \mathbb{Y}^c_n=\mathbb{Y}^c_0+c (k,n-k)\,|\, L^c=\alpha\big) =\binom{n}{k} \alpha^k (1-\alpha)^{n-k}=\P\big(\mathbb{Z}^c_{\alpha,n}=\mathbb{Z}^c_{\alpha,0}+c (k,n-k)\big)\, ,
\]
for every $k,n$ such that  $0 \le k \le n$ and all $\alpha \in (0,1)$,
and that
\begin{equation}\label{eq:mixture of random walks one step c}
\begin{aligned}
\P\big( \mathbb{Y}^c_n=\mathbb{Y}^c_0+c(k,n-k)\big)&=\int_0^1 \P\big(\mathbb{Z}_{\alpha,n}=\mathbb{Z}_{\alpha,0}+c(k,n-k)\big) \,f_{a_0/c,b_0/c}(\alpha) \,d\alpha
\\&=\int_0^1 \P(\textsc{bin}(n,\alpha)=k)\,f_{a_0/c,b_0/c}(\alpha) \,d\alpha \,,\end{aligned}\end{equation}
for every $k,n$ such that $0 \le k \le n$.

In other terms, \textit{P\'{o}lya's walk $ (\mathbb{Y}^c_n )_{n \ge 0}$ with steps of size $c$ starting from the initial position $(a_0,b_0)$ is a mixture of the $\alpha$-random walks $ (\mathbb{Z}^c_{\alpha,n} )_{n \ge 0}$, all starting at $(a_0,b_0)$, where the mixture parameter $\alpha \in (0,1)$ follows a $\textsc{Beta}(a_0/c,b_0/c)$ probability distribution.}

\subsection{An extension of Dirichlet's density result}\label{sec:extension 1 de Dirichlet}

The P\'{o}lya walk $\mathbb{Y}_n^c$ starting from  $\mathbb{Y}_0^c=(1,1)$ may only visit
 the points of the grid $\mathcal{G}_c$ of $\N^2$ given by
\[\mathcal{G}_c=\{(1+n c, 1+mc):  n,m\ge 0\}\,.\] The set $\mathcal{G}_c\cap \mathcal{V}$ is the set of visible points which the walk $\mathbb{Y}_n^c$ starting $ (1,1)$ can actually visit.

The density in $\N^2$ of the intersection $\mathcal{G}_c\cap \mathcal{V}$  (of visible and visitable points) is determined, as we are going to see in Proposition \ref{prop:density if step c},  by the quantity
\begin{equation}\label{eq:first def of Delta(c)}
\Delta(c)\triangleq \sum_{\substack{d \ge 1,\\ \gcd(d,c)=1}} \frac{\mu(d)}{d^2}\,\cdot
\end{equation}
Observe that $\Delta(1)=1/\zeta(2)$.
We may write, alternatively,
\begin{equation}\label{eq:second def of Delta(c)}
\Delta(c)=\prod_{p\,\nmid\, c}\Big(1-\frac{1}{p^2}\Big)= \frac{1}{\zeta(2)}\, \frac{1}{\prod_{p \mid c}\limits \big(1-{1}/{p^2}\big)}\, \cdot
\end{equation}
From 
 expression  \eqref{eq:second def of Delta(c)} we see that $\Delta(c)$ depends only on the prime factors of $c$, disregarding their multiplicity; moreover, it shows that $1\ge \Delta(c)\ge 1/\zeta(2)$, as it should.

\begin{prop} \label{prop:density if step c} For any integer $c \ge 1$, we have that
\begin{equation}\label{eq:density if step c}\lim_{N\to \infty}\frac{1}{N^2}\#\big\{0\le n,m\le N: (1+nc,1+mc)\in \mathcal{V}\big\}=\Delta(c)\,.\end{equation}
\end{prop}

The case $c=1$ is Dirichlet's density result of
Section \ref{seccion:Dirichlet's density}; the proof of Proposition~\ref{prop:density if step c}, below, follows the lines of the derivation in that section. 

Because of \eqref{eq:density if step c}, the density of $\mathcal{G}_c\cap \mathcal{V}$ is
\[
D(\mathcal{G}_c\cap \mathcal{V})=\frac{1}{c^2}\, \Delta(c).
\]
Since the density of $\mathcal{G}_c$ is $1/c^2$, we may write $\Delta(c)$ as the \textit{relative density}
\[
\Delta(c)=\frac{D(\mathcal{G}_c\cap \mathcal{V})}{D(\mathcal{G}_c)}\,, \quad \mbox{for $c\ge 1$}\,\cdot\]

\begin{proof}
Observe that for integer $N \ge 1$, we have that
\[
\begin{aligned}
&\#\{0\le n,m\le N: (1+nc,1+mc)\in \mathcal{V}\}=\#\{0\le n,m\le N: \gcd(1+nc,1+mc)=1\}
\\
&\qquad =\sum_{0\le n, m\le N}\delta_1(\gcd(1+nc,1+mc))
=\sum_{d \ge 1}\mu(d)\, \big[\#\{0\le j\le N: d\mid 1+jc\}\big]^2
\end{aligned}
\]

Note that, for $d\ge 1$, the number $\#\{0\le j\le N: d\mid 1+jc\}$ is 0 if $\gcd(d,c)>1$ or $d\ge Nc+2$.

Now, if $\gcd(d,c)=1$, the equation $1+xc\equiv 0 \mod d$ has a unique solution, and so 
\[
(\star) \qquad \frac{N}{d}-1\le \#\{0\le j\le N: d\mid 1+jc\}\le \frac{N}{d}+1\,.\]
If, moreover, $d\le Nc+1$, then
\[
(\star\star)\qquad
\frac{d}{N}\#\{0\le j\le N: d\mid 1+jc\}\le 1+ \frac{d}{N}\le   c+2\,.\]

Dominated convergence in conjunction with $(\star)$ and $(\star\star)$ gives us that
\begin{align*}
&\frac{1}{N^2}\sum_{\substack{d \ge 1;\\ \gcd(d,c)=1}}\mu(d)\, \big[\#\{0\le j\le N: d\mid 1+jc\}\big]^2
\\&\qquad =\sum_{\substack{d \ge 1; \gcd(d,c)=1}}\frac{\mu(d)}{d^2} \Big[\frac{d}{N}\#\{0\le j\le N: d\mid 1+jc\}\uno\nolimits_{\{1\le d \le 1+Nc\}}(d)\Big]^2
\end{align*}
tends to $\Delta(c)$ as $N\to \infty$.
\end{proof}

\subsection{Asymptotic average visibility of P\'{o}lya's of step $c$}\label{section:walk with step c}

Fix $c\ge 1$. Consider the P\'{o}lya walk $(\mathbb{Y}_n^c)_{n\ge 0}$ with steps of size $c\ge 1$. We denote with $Q_N^c$, for integer $N\ge 1$, the random variable
\[
Q_N^c= \frac{1}{N} \#\big\{1\le n \le N: I(\mathbb{Y}^c_n)=1\big\}=\frac{1}{N} \sum_{n=1}^N I(\mathbb{Y}_n^c)\,,\]
and also, for integer $N\ge 1$ and $\alpha \in (0,1)$, we write
\[
S_{\alpha,N}^c= \frac{1}{N} \#\big\{1\le n \le N: I(\mathbb{Z}^c_{\alpha,n})=1\big\}=\frac{1}{N} \sum_{n=1}^N I(\mathbb{Z}_{\alpha, n}^c)\,.\]

\

\subsubsection{Starting point $(a_0,b_0)=(1,1)$.} We first  consider walks (P{\' o}lya and standard)   \textit{starting from the initial position $(a_0,b_0)=(1,1)$} and we shall appeal to the density  result \eqref{eq:density if step c} of Proposition \ref{prop:density if step c}.

We have the following asymptotic result.

\begin{thm}\label{teor:polya from (1,1) step c, bis}
Let $c\ge 1$, be an integer. For P\'{o}lya's walk $(\mathbb{Y}^c_n)_{n \ge 0}$ with up-step $(0,c)$ and right-step $(c,0)$ starting from $(a_0,b_0)=(1,1)$, we have that
\[\lim_{N \to \infty} Q_N^c=\Delta(c)\, , \quad \mbox{almost surely}\, .\]
\end{thm}

This is (part of) Theorem~\ref{teor:polya from (1,1) step c} from the introduction: the case where $(a_0, b_0)=(1,1)$. In the notation there, $\Delta(1, 1; c,c)$ coincides with the (relative) density $\Delta(c)$ that we are dealing with in this section.

Recall that $\Delta(1)=1/\zeta(2)$, as it should, in accordance with Theorem~\ref{teor:asymptotic coprime proportion of polya walks}.  Notice moreover (see the very definition \eqref{eq:second def of Delta(c)} of $\Delta(c)$) that, the more prime factors the step $c$ has, the more time P\'{o}lya's walk with up and right-steps of size $c$ remains invisible.

By the way, for the $\alpha$-random walk $(\mathbb{Z}^c_{\alpha,n})_{n \ge 0}$  we also have
\[\lim_{N \to \infty} S^c_{\alpha, N}=\Delta(c)\, , \quad \mbox{almost surely}\, .\]

The proof of Theorem \ref{teor:polya from (1,1) step c}  follows the lines of the proof of Theorem \ref{teor:asymptotic coprime proportion of polya walks}. The main differences are listed below.

We use the following extension (and corollary) of Lemma \ref{lemma:binomial probabilities residue class}
\begin{lemma}\label{lemma:binomial probabilities residue class with c}
For integers $n\ge 1$, $ c,d\ge 1$ such that $\gcd(c,d)=1$, and $r\in\Z$,
there is an absolute constant $C>0$ such that, for any $\alpha\in(0,1)$, there holds
\begin{equation}\label{eq:binomial probabilities residue class with c}
\Big|\sum_{\substack{0\le l\le n;
\\
lc\,\equiv\, r{ \ \text{\rm mod}\, } d}}\binom nl \alpha^l(1-\alpha)^{n-l}-\frac1d\Big|\le \frac{C}{\sqrt{\alpha(1-\alpha)}}\, \frac{1}{\sqrt{n}}.
\end{equation}
\end{lemma}
\begin{proof}
Since $\gcd(c,d)=1$, (the residue class of) $c$ has an inverse $\beta$ in the group $\Z_d$; thus $\beta c\equiv 1$ mod $d$. So that
\[lc\equiv r{ \ \text{\rm mod}\, } d \iff l\equiv \beta r{ \ \text{\rm mod}\, } d\,.\] Applying Lemma \ref{lemma:binomial probabilities residue class} with $r$ replaced by $\beta r$, we get the result.
\end{proof}

Assume $(a_0,b_0)=(1,1)$.
For $\alpha\in (0,1)$ and integer $n \ge 1$, we have that
\[\begin{aligned}
\E(I(\mathbb{Z}^c_{\alpha, n}))&=\sum_{l=0}^n \binom{n}{l} \alpha^l (1-\alpha)^{n-l} \delta_1(\gcd(1+lc, 2+nc))\\
&=\sum_{l=0}^n \binom{n}{l} \alpha^l (1-\alpha)^{n-l} \sum_{\substack{d\mid 1+lc,\\d\mid 2+nc}}\mu(d)\,.\end{aligned}\]

Observe that if  $d\mid 1+lc$, for some $l \ge 0$, then $\gcd(c,d)=1$, and thus we may rewrite the expression above as
\[\begin{aligned}
\E(I(\mathbb{Z}^c_{\alpha, n}))=
\sum_{\substack{d\mid 2+nc;\\ \gcd(c,d)=1}} \mu(d)\sum_{\substack{0\le l \le n;\\ lc\equiv -1{ \ \text{\rm mod}\, } d}}\binom{n}{l} \alpha^l (1-\alpha)^{n-l}\,.\end{aligned}\]

Using Lemma \ref{lemma:binomial probabilities residue class with c} and arguing as in the proof of Lemma \ref{lemma:mean variance Zan}, we deduce that
\[
\E(I(\mathbb{Z}^c_{\alpha, n}))=\sum_{\substack{d\mid 2+nc;\\ \gcd(c,d)=1}}\frac{\mu(d)}{d}+\frac{1}{\sqrt{\alpha(1-\alpha)}} \, O\Big(\frac{1}{n^{1/4}}\Big)\,,\quad\text{as $n\to\infty$}.
\]

Analogously, if the walk starts at $(a_0,b_0)=(1+kc,1+qc)$, for some integers $k,q\ge 0$, we have
\[
\E(I(\mathbb{Z}^c_{\alpha, n}))=\sum_{\substack{d\mid 2+(k+q)c+nc;\\ \gcd(c,d)=1}}\frac{\mu(d)}{d}+\frac{1}{\sqrt{\alpha(1-\alpha)}} \,O\Big(\frac{1}{n^{1/4}}\Big)\,, \quad\text{as $n\to\infty$}.
\]
As in Proposition \ref{prop:bounds for mean and var of RW dependence on alpha},
 we deduce for starting point $(a_0,b_0)=(1+kc,1+qc)$, with $k,q\ge 0$, that
\[\E(S^c_{\alpha,N})=\frac{1}{N}\sum_{n=1}^N \E(I(\mathbb{Z}^c_{\alpha, n}))=\sum_{\substack{d\ge 1;\\ \gcd(c,d)=1}}\frac{\mu(d)}{d^2}+\frac{1}{\sqrt{\alpha(1-\alpha)}}\, O\Big(\frac{1}{n^{1/4}}\Big)\,,\quad\text{as $n\to\infty$.}\]

Now, since
\[
\int_0^1 \frac{1}{\sqrt{\alpha (1-\alpha)}} \, f_{1/c+k, 1/c+q}(\alpha) \, d \alpha<+\infty\,,
\]
if we start from a point $(a_0,b_0)=(1+kc,1+qc)$ with both $k,q\ge 1$, we deduce
that
\[\E(Q_N^c)=\Delta(c)+O\Big(\frac{1}{N^{1/4}}\Big)\,, \quad \mbox{as $N \to \infty$}\,.\]

Moreover, it can be shown that, starting from a point   $(a_0,b_0)=(1+kc,1+qc)$ with both $k,q\ge 1$,
\[\V(Q_N^c)=O\Big(\frac{1}{N^{1/4}}\Big)\,, \quad \mbox{as $N \to \infty$}\,.\]

The second moment method (Proposition \ref{prop:Borel}) then gives, as in the case $c=1$ and  starting from a point   $(a_0,b_0)=(1+kc,1+qc)$ with both $k,q\ge 1$, that
\[
\lim_{N\to \infty} Q_N^c=\Delta(c)\,, \quad \mbox{almost surely}\,.\]

To get the result of Theorem \ref{teor:polya from (1,1) step c, bis} (with starting point $(a_0,b_0)=(1,1)$), we just need to appeal to Remark \ref{remark: walk enters comfort zone}.

\begin{remark}\label{remark: walk enters comfort zone}
For P\'{o}lya's walk $\mathbb{Y}_n=((a_n,b_n))_{n \ge 0}$ starting at $(a_0,b_0)=(1,1)$ and with step of size $c\ge 2$, we have
\[\begin{aligned}
\P (\mbox{first $n$ steps are upwards} )&=\P(a_n = 1 + c n, b_n=1)\\&=\prod_{j=1}^n \frac{1+c(j-1)}{2+c(j-1)}
=\frac{\Gamma(1/c+n) \Gamma(2/c)}{\Gamma(1/c) \Gamma(2/c+n)}\,. \end{aligned}\]
Since, for $x \ge 0$,
\[
\Gamma(n+x)\sim \Gamma(n) n^x\,, \quad \mbox{as $n \to \infty$}\,,\] we see that
\[
\P(\mbox{first $n$ steps are upwards})\sim \Big(\frac{\Gamma(2/c)}{\Gamma(1/c)} \Big) \, \frac{1}{n^{1/c}}\,,\quad\mbox{as $n \to \infty$}\,.
\]
In particular,  $\P(\mbox{first $n$ steps are upwards})$ tends to 0 as $n\to \infty$. The same observations apply to the case where the first steps are rightwards.
\end{remark}

\subsubsection{General starting point $(a_0,b_0)$.} In the general case, when  the starting point $(a_0,b_0)$ of the P\'{o}lya walk $(\mathbb{Y}^c_n)_{n\ge 0}$  with up and right steps of size $c$ is not necessarily $(1,1)$, by appealing to the density result of Proposition \ref{prop:density general steps and starting point} instead of Proposition \ref{prop:density if step c} of Section~\ref{seccion:general density result} and arguing as in the case  $(a_0,b_0)=(1,1)$ above, we obtain the following.

\begin{thm}\label{teor:polya from (a_0,b_0) steps generales}
Fix integers $a_0, b_0\ge1$ and $c\ge 1$ and consider the  P\'{o}lya walk $(\mathbb{Y}_n)_{n \ge 0}$ with up-step $(0,c)$ and right-step $(c,0)$ starting from $(a_0,b_0)$. Then we have that
\[\lim_{N \to \infty} Q^c_N=\Delta(a_0,b_0;c,c)\, , \quad \mbox{almost surely}\, .\]
\end{thm}
The  almost sure limit $\Delta(a_0,b_0;c,c)$ is given by the expression \eqref{eq:formula for big Delta con mu y c} below. 

\subsection{A further extension of Dirichlet's density result}\label{seccion:general density result} For integers $a_0, b_0, r_0, u_0\ge1$, we denote
\begin{equation}\label{eq:formula for big Delta con mu}
\Delta(a_0,b_0;r_0,u_0)=\sum_{\substack{d \ge 1;\\ \gcd(d,r_0)\mid a_0;\\ \gcd(d,u_0)\mid b_0}}
\frac{\mu(d)}{d^2}\gcd(d,r_0)\gcd(d,u_0)\, .
\end{equation}
Below, see \eqref{eq:formula for big Delta}, we exhibit an alternative expression for $\Delta(a_0,b_0;r_0,u_0)$ in terms of primes numbers dividing or not the parameters $a_0$, $b_0$, $  r_0$ and $u_0$.

As we have already mentioned, for each integer $c \ge 1$, $\Delta(1, 1; c,c)$ equals the (relative) density $\Delta(c)$ of Section~\ref{sec:extension 1 de Dirichlet}.
Observe also that $\Delta(a_0,b_0;1,1)=1/\zeta(2)$  and, more generally, that
\begin{equation}\label{eq:formula for big Delta con mu y c}
\Delta(a_0,b_0;c,c)=\sum_{\substack{d \ge 1;\\ \gcd(d,c)\mid \gcd(a_0,b_0)}}
\frac{\mu(d)}{d^2}\gcd(d,c)^2\, .
\end{equation}
But  see also alternatively \eqref{eq:formula for big Delta con mu y c bis}.


%

Analogously as in Proposition  \ref{prop:density if step c}, we have the following density result.
\begin{prop}\label{prop:density general steps and starting point}
For integers $a_0, b_0, r_0, u_0\ge1$,
\begin{equation}\label{eq:density general steps and starting point}
\lim_{N\to \infty}\frac{1}{N^2}\#\big\{0\le n,m\le N: (a_0+n r_0,b_0+m u_0)\in \mathcal{V}\big\}=\Delta(a_0, b_0, r_0, u_0)\,.\end{equation}
\end{prop}

%

%

\

Next we discuss an alternative expression \eqref{eq:formula for big Delta} of $\Delta(a_0,b_0;r_0,u_0)$ in terms of primes. Let us denote by $A$ the set of primes that divide $a_0$; write analogously $B$, $ R$ and $U$ for the sets of primes dividing $b_0$, $r_0$ and $u_0$, respectively. Finally, denote $H=((R\setminus U)\cap A)\cup ( (U\setminus R)\cap B)$. Then we have that
\begin{equation}\label{eq:formula for big Delta}
\Delta(a_0, b_0; r_0, u_0)=\delta_1(\gcd(a_0,b_0;r_0,u_0))
\times \prod_{p \in H} \Big(1-\frac{1}{p}\Big)
\times \prod_{p \,\nmid \,{\rm lcm}(r_0,u_0)}\Big(1-\frac{1}{p^2}\Big)\,.
\end{equation}

This gives in particular that
\begin{equation}\label{eq:formula for big Delta con mu y c bis}\Delta(a_0,b_0;c,c)=\delta_1(\gcd(a_0,b_0,c))\, \Delta(c)\,,\end{equation}
so in the case of common jump $c$ upwards and rightwards, the density is either $0$ or~$\Delta(c)$.

\begin{proof}[Proof of the representation \eqref{eq:formula for big Delta}]
Since $\mu(d)$ is 0 unless $d$ is 1 or a product of distinct primes, $\Delta(a_0,b_0;r_0,u_0)$ depends only of the primes that divide each of the parameters $a_0, b_0, r_0, u_0$, and not on their multiplicities as divisors. We may then assume that the numbers~$d$ in the sum defining $\Delta(a_0,b_0;r_0,u_0)$, aside from 1,  are all products of distinct  primes. And thus, comparing both sides of \eqref{eq:formula for big Delta}, we may assume that $a_0, b_0, r_0, u_0$ are all products of distinct primes, or 1.

For a set $Q$ of prime numbers, we let $\Pi(Q)$ be the set consisting of 1 and all the products of distinct primes extracted from $Q$.
Recall that, if $g$ is a multiplicative function, then
\[
(\flat)\qquad \sum_{d \in \Pi(Q)}\mu(d) \,g(d)=\prod_{p\in Q}(1-g(p))\,.\]

The sets of primes $A$ and $B$ determine the partition of $\N$ consisting of the 4 blocks $A\setminus B$, $ B\setminus A$, $ A\cap B$ and $E=\N\setminus (A\cup B)$. Some of these blocks could be empty, for instance $A\cap B$ is empty if $\gcd(a_0,b_0)=1$, and $A \setminus B$ is empty if $a_0 \mid b_0$. Moreover, the sets $A$ or $B$ are empty if $a_0=1$ or $b_0=1$.

Likewise, the sets of primes $R$ and $U$  determine the partition of $\N$ consisting of $R\setminus U$, $ U\setminus R$, $ R\cap U$ and $W=\N\setminus (R\cup U)$.

The common refinement of these two partitions is a partition of $\N$ with 16 (pairwise disjoint) blocks, some of which could be, of course, empty.

Let $\mathcal{D}$ denote the subset of $\N$ given by
\[
\mathcal{D}=\{d\ge 1: \mbox{$d$ is 1 or a product of distinct primes}, \, \, \gcd(d,r_0)|a_0 \,\, \mbox{and} \,\,\gcd(d,u_0)|b_0 \}.
\]
This is the set of $d$'s that conform the sum in \eqref{eq:formula for big Delta con mu}.

Fix $d\in \mathcal{D}$. And suppose, as illustration, that a prime $p$ belongs to $R\setminus U$. That is, $p\mid r_0$ but $p\nmid u_0$. Then, in order to be a prime factor of $d$, we must have $p\mid a_0$, but there is no additional restriction related to $b_0$. (This argument explains blocks $C_1$ and $C_2$ in the display below.)

Arguing analogously with the remaining blocks, one discovers that the prime factors of any $d \in \mathcal{D}$ must belong to the following 6 blocks of the refined partition:
\[\begin{aligned}
C_1&=(R\setminus U)\cap (A\setminus B)\,,  &&C_2=(R\setminus U)\cap (A\cap B)\,,\\
C_3&=(U\setminus R)\cap (B\setminus A)\,,  &&C_4=(U\setminus R)\cap (A\cap B)\,,\\
C_0&=(R\cap U)\cap (A\cap B)\,, &&C_5=W\,.
\end{aligned}
\]
In fact, and conversely, any integer which is a product of distinct primes extracted from $\bigcup_{j=0}^5 C_j$ is in $\mathcal{D}$.

Observe that the set $H$ of \eqref{eq:formula for big Delta} is $H=(C_1\cup C_2)\cup (C_3\cup C_4)$.

Let $d\in \mathcal{D}$ be written as $d=\prod_{j=0}^5 d_j$, where each $d_j$ is either $1$ or a product of distinct primes extracted for $C_j$, i.e., $d_j\in \Pi(C_j)$. This factorization is unique, as the blocks are pairwise disjoint. We have then that
\[
\gcd(d,r_0)=d_1 d_2 d_0 \quad \mbox{and} \quad \gcd(d,u_0)=d_3d_4 d_0
\] and thus
\[
\frac{\mu(d)}{d^2}\gcd(d,r_0)\gcd(d,u_0)=\frac{\mu(d_0)\cdot \mu(d_1)\cdots \mu(d_5)}{d_1\cdots d_4\cdot d_{5}^2 }\,\cdot
\]

Therefore, using $(\flat)$ and the last equality, we have that
\begin{align*}
\sum_{d\in \mathcal{D}}\frac{\mu(d)}{d^2}\gcd(d,r_0)&\gcd(d,u_0)=
\sum_{d_j\in \Pi(C_j); 0\le j \le 5} \frac{\mu(d_0)\cdot \mu(d_1)\cdots \mu(d_5)}{d_1\cdots d_4\cdot d_{5}^2 }\\&=
\Big(\sum_{d\in \Pi(C_0)}\mu(d)\Big)
\Big(\sum_{d\in \Pi(H)}\frac{\mu(d)}{d}\Big)
\Big(\sum_{d\in \Pi(W)}\frac{\mu(d)}{d^2}\Big)\\&=
\delta_1(\gcd(a_0,b_0, r_0, u_0))\prod_{p \in H} \Big(1-\frac{1}{p}\Big)
\prod_{p \,\nmid \,\text{lcm}(r_0,u_0)}\Big(1-\frac{1}{p^2}\Big)\,.\qedhere
\end{align*}
\end{proof}

\section{Questions on more general walks}

1) \textit{Unequal steps}. We may consider a P\'{o}lya walk with rightwards step of size $r_0$ and upwards step of size $u_0$, not necessarily equal, starting from a general point  $(a_0,b_0)\in \N^2$.
We already have a general density result (see Proposition \ref{prop:density general steps and starting point}, and the formulas~\eqref{eq:density general steps and starting point} or~\eqref{eq:formula for big Delta}) for the visible points which such a walk may visit.

This general P\'{o}lya walk corresponds to a P\'{o}lya urn process where, at each stage, $r_0$ amber balls are added if the observed/drawn  ball is amber and $u_0$ blue balls are added if the observed/drawn ball if blue, and where at the beginning   the urn contains $a_0$ amber balls and $b_0$ blue balls.

But in this general situation the sequence of Bernoulli variables $F^{r_0,u_0}_n$ registering whether the $n$-th step is to the right, $F^{r_0,u_0}_n=1$, or up, $F^{r_0,u_0}_n=0$, is \textit{not exchangeable}. In fact, if $(x_1, \ldots, x_n)$ is a list extracted from $\{0,1\}$ and if $t_l=\sum_{j=1}^l x_j$, for $1\le l \le n$, then
\[
\P(F^{r_0,u_0}_1=x_1, \ldots, F^{r_0,u_0}_n=x_n)=\frac{\prod_{j=0}^{t_n-1} (a_0+jr_0)\, \prod_{j=0}^{n-t_n-1} (b_0+ju_0)}{\prod_{j=0}^{n-1} (a_0+b_0+ju_0+(r_0-u_0) t_j)}\,.\]
If $r_0\neq u_0$, the probability above depends not only on $t_n$, the number of $x_j$ which are equal to 1, but on the $t_j$, $1\le j \le n{-}1$, and thus on the order of appearance of the 1's in the list of $x_j$. Actually, the sequence of Bernoulli variables $F^{r_0,u_0}_n$ is exchangeable if and only if $r_0=u_0$, i.e., when the sizes of the rightwards and upwards steps coincide. This is the case of the walk $\mathbb{Y}_n^c$ which we have discussed in Section \ref{section:walk with step c}.

But still, with notation with obvious meaning at this time, one wonders if
\[\lim_{N \to \infty} Q^{r_0,u_0}_N=\Delta(a_0,b_0;r_0, u_0)\, , \quad \mbox{almost surely}\, .\]

\smallskip
(2) \textit{Compensated P\'{o}lya's urn}. What happens in the case when, at each stage of the urn, the ball added is of the other color than the ball drawn? This is an instance of the so called Bernard Friedman's urn, see \cite{Friedman}, \cite{Freedman}. Observe that, with the general notations above, in this model $a_n/(a_n+b_n)$ tends to $1/2$ almost surely.  But, again, the Bernoulli variables registering the rightwards and upwards movements are  \textit{not} exchangeable. Is it  the case that the average visibility  time converges to $1/\zeta(2)$ no matter what the starting point is?

\smallskip
(3)
\textit{Basic three dimensional extension}. Start with an urn with balls of  3 colors: $a,b,c$. The process now is to draw a ball form the urn, note the color of this ball, return it to the urn and  add one ball of the same color to the urn. This determines a corresponding P\'{o}lya's walk in $\{1, 2, \ldots\}^3$. There are now two notions of coprimality for the composition $(a_n,b_n,c_n)$ of the urn at time $n$: coprime  triple, that is, $\gcd(a_n,b_n,c_n)=1$, or pairwise coprime, i.e., $\gcd(a_n,b_n)=\gcd(a_n,c_n)=\gcd(b_n,c_n)=1$. Is it the case that, for a certain  constant $P$, the  proportion of time up to time $N$ that the composition of the urn is a coprime triple converges almost surely to $P$ as $N\to \infty$? For pairwise coprimality, does the same result hold with a certain constant $T$? The constants $P$ and $T$ should be    
\[
P=\prod_{p}\Big(1-\frac{1}{p^3}\Big)=\frac{1}{\zeta(3)} \quad \mbox{and}\quad T=
\prod_{p}\Big(1-3\, \frac{1-1/p}{p^2}-\frac{1}{p^3}\Big).
\]
The number $T$ is the asymptotic proportion of triples of integers that are pairwise coprime (see T\'{o}th \cite{To2004} and Cai--Bach\cite{CB2001}, Theorem 5).

\smallskip

(4) A \textit{general P\'{o}lya urn} has a starting composition of numbers of balls of a finite collection of, say $m$,  different colors $\alpha_1, \alpha_2, \ldots, \alpha_m$, and a $m \times m$ matrix with nonnegative integer entries which codifies how many balls of each color are added of each color at each stage depending of the color of the drawn ball. The evolution of this urn determines a walk in $\{1, 2, \ldots\}^m$.  What happens in this case?

\end{document}